%% file: Trigonal.tex
\documentclass[11pt]{article}

\usepackage{ulem}
\usepackage{framed}

\input{preamble.tex}

\makeatletter
\def\@seccntformat#1{\csname the#1\endcsname. }
\makeatother
\makeatletter
\renewcommand\section{\@startsection {section}{1}{\z@}%
 {-3.5ex \@plus -1ex \@minus -.2ex}%
 {2.3ex \@plus.2ex}%
 {\normalfont\large\bfseries}}
\makeatother

\begin{document}

\title
{\bf Superspecial trigonal curves of genus $5$}
\author
{Momonari Kudo\thanks{Kobe City College of Technology} \thanks{Institute of Mathematics for Industry, Kyushu University.
E-mail: \texttt{m-kudo@math.kyushu-u.ac.jp}}
\ and Shushi Harashita\thanks{Graduate School of Environment and Information Sciences, Yokohama National University.
E-mail: \texttt{harasita@ynu.ac.jp}}}

\maketitle

\input{section1.tex}
\input{section2.tex}

\input{section3.tex}
\input{section4.tex}
\input{section5.tex}
\input{section6.tex}

\end{document}

%% file: preamble.tex
\usepackage{amsthm}
\usepackage{amssymb}
\usepackage{amscd}
\usepackage{amsmath}
\usepackage[all]{xy}
\usepackage[top=30truemm,bottom=30truemm,left=25truemm,right=25truemm]{geometry}
\usepackage{comment}
\usepackage{algorithmic,algorithm}
\usepackage{here}
\usepackage{graphicx}
\usepackage{color} %
\usepackage{enumerate} %
\usepackage{listliketab}
\usepackage{multirow}


\makeatletter
  
  \@addtoreset{algorithm}{subsection}
\makeatother

\theoremstyle{definition}

\theoremstyle{definition}

\theoremstyle{plain}
\newtheorem{theo}{Theorem}

\theoremstyle{plain}
\newtheorem{theor}{Theorem}

\theoremstyle{plain}

\theoremstyle{plain}

\theoremstyle{plain}
\newtheorem{thm}{Theorem}[subsection]

\theoremstyle{definition}

\theoremstyle{definition}

\theoremstyle{definition}

\theoremstyle{definition}

\theoremstyle{definition}
\newtheorem{rem}[thm]{Remark}

\theoremstyle{plain}
\newtheorem{prop}[thm]{Proposition}

\theoremstyle{plain}
\newtheorem{lem}[thm]{Lemma}

\theoremstyle{plain}
\newtheorem{cor}[thm]{Corollary}

\theoremstyle{definition}

\theoremstyle{definition}

\theoremstyle{definition}

\theoremstyle{definition}

\theoremstyle{definition}

\numberwithin{equation}{subsection}

\def\F{\mathbb{F}}

\def\deg{{\rm deg}}

\def\ord{{\rm ord}}

\def\dim{{\rm dim}}

\newcommand{\bbP}{{\operatorname{\bf P}}}

\newcommand{\Aut}{\operatorname{Aut}}

\newcommand{\modulo}{\operatorname{mod}}

%% file: section1.tex
\begin{abstract}
This paper provides an algorithm enumerating superspecial trigonal curves of genus $5$
over finite fields.
Executing the algorithm over a computer algebra system Magma, we enumerate them over finite fields $\mathbb{F}_{p^a}$ for any natural number $a$ if $p \leq 7$ and for odd $a$ if $p \leq 13$.
\end{abstract}

\section{Introduction}\label{sec:Intro}
Let $p$ be a rational prime.
Let $K$ be a perfect field of characteristic $p$.
Let $\overline{K}$ denote the algebraic closure of $K$.
By a curve, we mean a nonsingular projective variety of dimension $1$.
A curve over $K$ is called {\it superspecial} if its Jacobian
is isomorphic to a product of supersingular elliptic curves over $\overline{K}$.
It is known that $C$ is superspecial if and only if the Frobenius on the first cohomology group $H^1(C, \mathcal{O}_C)$ is zero, where $\mathcal{O}_C$ is the structure sheaf of $C$.
As superspecial curves in characteristic $p$ often descend to maximal curves over $\mathbb{F}_{p^2}$, they are important objects for example in coding theory.

In the series of papers \cite{KH16}, \cite{KH17}, \cite{KH17a} and \cite{KHS17},
we enumerate superspecial curves of genus 4 over small finite fields
and study their automorphism groups.
This paper is the first attempt to obtain an analogous result for genus $5$.
As the whole moduli space of curves of genus $5$ is so huge,
we here restrict ourselves to {\it trigonal} ones.
Here, a curve $C$ is said to be trigonal if there exists a morphism $C \to \mathbf{P}^1$ of degree $3$.

In \cite[Theorem 1.1]{Ekedahl}, Ekedahl proved that $2g \leq p^2-p$ holds if a superspecial curve $C$ of genus $g$ in characteristic $p$ exists.
This implies that there is no superspecial curve of genus $5$ in characteristic $p=2$ ,$3$.
The non-existence holds also for $p=5$. Indeed, by \cite[Lemma 2.2.1]{KH16},
if there were a superspecial curve of genus $5$ in characteristic $5$,
then there would exist a maximal curve of genus $5$ over $\mathbb{F}_{25}$,
but this contradicts the fact due to Fuhrmann and Torres \cite{FT}
that if there exists a maximal curve of genus $g$ over $\mathbb{F}_{p^2}$,
then $4g \leq (p-1)^2$ or $2g=p^2-p$.

In this paper, we enumerate superspecial trigonal curves of genus $g=5$ over finite fields $\mathbb{F}_{p^a}$ for any $a$ if $p \leq 7$
and for odd $a$ if $p \leq 13$.
Since the number of isomorphism classes of superspecial curves
depends on the parity of $a$ (cf.\ \cite[Proposition 2.3.1]{KH17}), it suffices to study the case of $a=1$, $2$ if $p \leq 7$ and the case of $a=1$ if $p \leq 13$.
The main theorems are as follows.
We state them separatedly in characteristic $p=7$, $11$ and $13$ respectively.

\begin{theo}\label{MainTheorem_q7}
There is no superspecial trigonal curve of genus $5$ in characteristic $7$.
\end{theo}

\begin{theo}\label{MainTheorem}
Any superspecial trigonal curve of genus $5$ over $\mathbb{F}_{11}$ is $\mathbb{F}_{11}$-isomorphic to the desingularization of
\begin{equation}
x y z^3 + a_1 x^5 + a_2 y^5  = 0 \label{sscurve_F11}
\end{equation}
in $\bbP^2$, where $a_1, a_2 \in \mathbb{F}_{11}^{\times}$, or the desingularization of
\begin{equation}
 (x^2 - 2 y^2) z^3 + a x^5 + b x^4 y + (9 a) x^3 y^2 + 4 b x^2 y^3 + ( 9 a ) x y^4 + 3 b y^5 =0 \label{sscurve_F11_2}
\end{equation}
in $\bbP^2$, where $(a, b) \in (\mathbb{F}_{11})^{\oplus 2} \smallsetminus \{ (0,0 )\}$.
\end{theo}
Based on Theorem B, we can find complete representatives of $\mathbb{F}_{11}$-isomorphism classes (resp. $\overline{\mathbb{F}_{11}}$-isomorphism classes)
of superspecial trigonal curves of genus $5$ over $\mathbb{F}_{11}$,
see Proposition \ref{prop:isom}.

\begin{theo}\label{MainTheorem_q13}
There is no superspecial trigonal curve of genus $5$ over $\mathbb{F}_{13}$.
\end{theo}

Here, let us describe our strategy to prove Theorems \ref{MainTheorem_q7}, \ref{MainTheorem} and \ref{MainTheorem_q13}.
First, we parametrize trigonal nonsingular curves of genus $5$ by quintic plane curves.
Concretely, it is shown that each trigonal curve $C$ of genus $5$ over $K$ is the desingularization of a quintic $C^{\prime}=V(F)$ with a unique singular point for some quintic form $F \in K [x,y,z]$.
It is also shown that the quintic forms defining our curves are divided into three types.

Second, we present a criterion for the superspeciality of trigonal curves of genus $5$.
Specifically, we give a method to compute their Hasse-Witt matrices, which are representation matrices for the Frobenius on the first cohomology group $H^1 ( C, \mathcal{O}_C )$.
As we have an explicit defining equation $F(x,y,z)=0$ of the singular model $C^{\prime}$, we can compute its Hasse-Witt matrix by known methods, e.g., \cite[Section 5]{Kudo}.
However, there is a gap between the Frobenius on $H^1 ( C^{\prime}, \mathcal{O}_{C^{\prime}} )$ and that on $H^1 ( C, \mathcal{O}_C )$ since $H^1 ( C, \mathcal{O}_{C} )$ is realized as a quotient space of $H^1 ( C^{\prime}, \mathcal{O}_{C^{\prime}} ) \cong H^2(\bbP^2, {\mathcal O}_{\bbP^2}(-5))$.
To solve this problem, we give an explicit basis of the quotient space of $H^2(\bbP^2, {\mathcal O}_{\bbP^2}(-5))$, and represent its image elements by the Frobenius as linear combinations of the basis.
Thanks to these explicit representations, it is proved that the Hasse-Witt matrix of $C$ is determined from certain coefficients in $F^{p-1}$.
This is an analogous result to previous works on computing Hasse-Witt matrices or Cartier-Manin matrices of algebraic curves (cf.\ \cite{Yui} for the Cartier-Manin matrices of hyperelliptic curves, and \cite[Section 4]{Kudo2} for the Hasse-Witt matrices of complete intersections).

Third, we give an algorithm for enumerating singular models $V(F)$ of superspecial trigonal curves of genus $5$.
Regarding unknown coefficients of quintic forms $F$ as indeterminates, we reduce enumerating them into solving multivariate systems over $K$ derived from our criterion for the superspeciality.
To solve the multivariate systems, we apply the {\it hybrid method}~\cite{BFP}, which mixes the brute-force on some coefficients with Gr\"{o}bner bases techniques.
In fact, our algorithm doubly uses the brute-force on coefficients, similarly to the algorithms for the genus $4$-case given in \cite{KH17} and \cite{KH17a}.
For each root, we substitute it into unknown coefficients of $F$, and decide whether $C^{\prime} = V( F)$ has a unique singularity or not.
In this way, one can collect quintic curves $V(F)$ such that their desingularizations are superspecial trigonal curves of genus $5$.

Moreover, we also give reductions of the defining equations of the singular models of $C$.
Since the number of variables deeply affects the efficiency of solving multivariate systems (see, e.g., \cite{Ayad}), reducing the number of unknown coefficients in $F$ is quite important for our computational enumeration to finish in practical time.
Using techniques similar to the reduction methods in \cite{KH16}, \cite{KH17} and \cite{KH17a}, we shall reduce the number as much as possible for each type of the quintic forms.

Executing our enumeration algorithm\footnote{We implemented the algorithm on Magma V2.22-7. All the source codes and log files are available at the web page of the first author \cite{HPkudo}.} over Magma~\cite{Magma}, we have succeeded in finishing all the computation to enumerate all superspecial trigonal curves of genus $5$ over $\mathbb{F}_q$ for $q=49$, $11$ and $13$.

This paper is organized as follows.
In Section \ref{section:2}, we recall some basic facts on trigonal curves
and study a (singular) model $C^{\prime}$ in $\bbP^2$ of each trigonal curve $C$ of genus $5$.
After we realize $H^1(C,{\mathcal O}_C)$ as a quotient of
$H^2(\bbP^2, {\mathcal O}_{\bbP^2}(-5))$,
we present a method to determine the Hasse-Witt matrix of $C$,
in particular we get an algorithm determining whether $C$
 is superspecial or not.
In Section \ref{SectionReduction}, we study reductions of the defining equations of the models in $\bbP^2$ obtained in Section \ref{section:2} in order to reduce the number of
unknown coefficients of defining equations.
This step is indispensable for making our computational enumeration efficient.
After these preparations, the main theorems will be proved in Section \ref{sec:main_results}.
Finally we study in Section \ref{sec:aut} the automorphism groups of
the superspecial trigonal curves obtained in the main theorem(s)
and the compatibility with the enumeration by the Galois cohomology theory.

\subsection*{Acknowledgments}
This work was supported by JSPS Grant-in-Aid for Scientific Research (C) 17K05196.

%% file: section2.tex
\section{Trigonal curves of genus $5$}\label{section:2}
In this section, we parametrize trigonal
nonsingular curves of genus $5$ by quintic plane curves,
preserving the base field.
In the last subsection,
we give an algorithm determining whether $C$ is superspecial,
depending on the type of the defining equation of each quintic plane curve.

\subsection{General facts on trigonal curves}\label{UniquenessPencil}
Let $C$ be a trigonal nonsingular curve of genus $g$.
Let $\pi: C \to \bbP^1$ be a morphism of degree $3$.
For the reader's convenience, we here recall a proof of the fact that
$C$ is not hyperelliptic if $g\ge 3$ and
the pencil $\pi$ is unique if $g\ge 5$.

Let $\rho: C\to\bbP^1$ be a morphism of degree $d=2$ or $3$.
Note that $\rho$ is associated to
 a linear subspace $M \subset H^0(C,{\mathcal M})$ with $\dim M = 2$
for some invertible sheaf $\mathcal M$ on $C$ of degree $d$.
If $g\ge 3$, then $M$ has to be complete, i.e., $M = H^0(C,{\mathcal M})$
by  the inequality of Clifford's theorem
$h^0(C,{\mathcal M}) -1 \le \deg({\mathcal M})/2$ (cf. \cite[Theorem 5.4]{Har}),
since $\mathcal M$ is special (cf. \cite[Example 1.3.4]{Har}) by Riemann-Roch
\[
h^0(\omega_C\otimes {\mathcal M}^{-1})
=h^0(C,{\mathcal M}) - \deg({\mathcal M}) -1 + g 
> 0.
\]
Let $\mathcal L$ be the invertible sheaf on $C$
making the pencil $\pi: C \to \bbP^1$ chosen first.


Suppose that
$\mathcal M$ is distinct from $\mathcal L$ (if $d=3$)
and that $g \ge 3$ if $d=2$ and that $g\ge 5$ if $d=3$.
We shall show that the existence of such a $\mathcal M$ makes a contradiction.
Let $s_1, s_2$ be linearly independent sections of ${\mathcal L}$.
Applying the Base-Point-Free Pencil Trick \cite[Chap.~III, \S~3, p.126]{ACGH} to $\mathcal L$, $\mathcal M$ and $s_1,s_2$,
we have an exact sequence
\[
\begin{CD}
0 @>>> H^0(C,{\mathcal L}^{-1}\otimes{\mathcal M}) @>>> H^0(C,{\mathcal M})^{\oplus 2}@>(s_1,s_2)>>H^0(C,{\mathcal L}\otimes{\mathcal M}).
\end{CD}
\]
As $\mathcal L$ and $\mathcal M$ are distinct, we get $h^0(C,{\mathcal L}^{-1}\otimes{\mathcal M})=0$ and therefore $h^0(C,{\mathcal L}\otimes{\mathcal M}) \ge 4$,
and in particular ${\mathcal L}\otimes{\mathcal M}$ is special by Riemann-Roch.
Then by the inequality of Clifford's theorem 
again, we have
\begin{equation}\label{CliffordThm}
h^0(C,{\mathcal L}\otimes{\mathcal M})-1 \le \frac{3+d}{2}.
\end{equation}
If $d=2$, then this is absurd, whence
there does not exist such a $\mathcal M$, i.e., $C$ is not hyperelliptic.
Also if $d=3$, so is this.
Indeed by Clifford's theorem 
the inequality \eqref{CliffordThm} has to be strict, since $C$ is nonhyperelliptic and
${\mathcal L}\otimes{\mathcal M}$ is
not the structure sheaf nor the canonical sheaf by looking at its degree.


\subsection{Quintic models}\label{TrigonalCurvesGenus5}
Let $p$ be a rational prime.
Assume $p\ne 2$.
Let $K$ be a field of characteristic $p$.
To each trigonal curve $C$ over $K$,
we can associate a quintic in $\bbP^2_K$.

\begin{lem}\label{CharacterizationTrigonal}
\begin{enumerate}
\item[\rm (1)]
Let $C$ be a trigonal curve of genus $5$ over $K$.
Then $C$ is the desingularization of
a quintic $C'$ in $\bbP^2_K$ 
such that $C'$ has a unique singular point, which is a $K$-rational point and has multiplicity $2$, i.e., the singular point is either a node or a cusp.
Moreover
\begin{enumerate}
\item[\rm (i)]
If $C'$ has a node,
$C'$ can be written as $V(F)$ in $\bbP^2$
for a single irreducible polynomial
\[
F =
\begin{cases}
 xyz^3 + f \qquad\text{or}\\
 (x^2-\epsilon y^2)z^3 + f,
\end{cases}
\]
where $\epsilon \in K \smallsetminus (K^\times)^2$
and $f$ is the sum of monomial terms which can not be divided by $z^3$.
\item[\rm (ii)]
If $C'$ has a cusp,
$C'$ can be written as $V(F)$ in $\bbP^2$
for a single irreducible polynomial
\[
F =
 x^2z^3 + f,
\]
where $f$
is the sum of monomial terms which can not be divided by $z^3$ and
$f$ contains non-zero $y^3z^2$-term.
\end{enumerate}
\item[\rm (2)]
Conversely, for $C'$ as in {\rm (i)} or {\rm (ii)}
above with a single singular point, its desingularization is a trigonal curve of genus $5$.
\item[\rm (3)]
There is a canonical isomorphism
from $\Aut_K(C)$ to $\{\mu \in \Aut_K(\bbP^2) \mid \mu(C') = C'\}$.
\end{enumerate}
\end{lem}
\begin{proof}
(1) Let $C$ be a nonsingular trigonal curve of genus $5$ over $K$.
Let $\mathcal L$ be an invertible sheaf making a pencil $C \to \bbP^1$ of degree $3$.
The uniqueness of such $\mathcal L$ (Section \ref{UniquenessPencil}) implies that
$\mathcal L$ is defined over the ground field $K$.

Let $\omega_C$ denote the canonical sheaf on $C$.
We have $\deg(\omega_C\otimes {\mathcal L}^{-1}) = (2g-2)-3 = 5$ and Riemann-Roch for ${\mathcal L}$ shows
$h^0(\omega_C\otimes {\mathcal L}^{-1})=h^0({\mathcal L})-\deg({\mathcal L})-1+g=3$.
Moreover $\omega_C\otimes {\mathcal L}^{-1}$ has no base-point.
Indeed, otherwise, let $\mathcal M$ be the invertible sheaf whose divisor is obtained
by subtracting the base-point from the divisor of $\omega_C\otimes {\mathcal L}^{-1}$, then $\deg(\mathcal M)=4$ but still $h^0({\mathcal M})=3$.
This contradicts Clifford's theorem,
as $C$ is nonhyperelliptic and $\mathcal M$ is special but is not $\omega_C$ nor $\mathcal O_C$.
The linear system of $\omega_C\otimes {\mathcal L}^{-1}$ defines
$\varphi: C \to \bbP^2$.

Let $C'$ denote the image of $\varphi$.
Note that $C'$ is a curve of degree $5$ in $\bbP^2$.
Let $\delta$ be the cokernel of ${\mathcal O}_{C'} \to {\mathcal O}_{C}$.
The short exact sequence
\begin{equation}\label{StandardExactSquence}
\begin{CD}
0 @>>> {\mathcal O}_{C'} @>>> {\mathcal O}_{C} @>>> \delta @>>> 0
\end{CD}
\end{equation}
gives a basic exact short sequence
\begin{equation}\label{StandardExactSquenceCoh}
\begin{CD}
0 @>>> H^0(\delta) @>>> H^1(C',{\mathcal O}_{C'}) @>>> H^1(C,{\mathcal O}_C) @>>> 0.
\end{CD}
\end{equation}

Recall the general fact (cf. \cite[8.3, Corollary 1. on p.\,103]{Fulton})
that for any plane curve $C$ of degree $d$ and geometric genus g we have
\[
g \le \frac{(d-1)(d-2)}{2} - \sum_P \frac{m_P(m_P-1)}{2}
\]
where $m_P$ is the multiplicity of $C'$ at $P$ (cf. \cite[3.1, p.\,32]{Fulton}).
This implies that $C'$ has a single singular point having
multiplicity $2$,
since otherwise the geometric genus of $C'$ less than $5$.
By using a $K$-automorphism of $\bbP^2$,
the singular point is assumed to be at $(0,0,1)$ and therefore
an equation $F$ defining $C'$ has to be of the form
\[
F = Qz^3 + f,
\]
where $Q$ is a quadratic form in $x,y$ and
$f$ is the sum of monomial terms which can not be divided by $z^3$.

If $Q$ is non-degenerate (node case),
there exists a linear transform in $x, y$ which sends $Q$
to $xy$ in the split case and to $x^2-\epsilon y^2$ in the non-split case
for some $\epsilon \in K \smallsetminus (K^\times)^2$.

If $Q$ is degenerate (cusp case),
there exists a linear transform in $x, y$ which sends $Q$
to $x^2$.
Note that $C$ is normalization of $C'$.
If the coefficient of $y^3z^2$ were zero, then
one could write
\[
F=x^2 g + xy^2 u + y^4 v,
\]
where $g,u,v\in K[x,y,z]$ with $y \nmid g$.
Since 
\[
\left(\frac{xg}{y^2}\right)^2 + u \left(\frac{xg}{y^2}\right) + v g = \frac{g}{y^4} F
\]
is zero on $C'$, the cokernel $\delta$ in \eqref{StandardExactSquence} contains
the two-dimensional space generated by the images of
$xg/y^2$ and $xg/y$.
Then the genus of $C$ is less than or equal to $4$
by \eqref{StandardExactSquenceCoh}.

Assume the coefficient of $y^3z^2$ is not zero.
Then at least two of $x^2$, $xy^2$ or $y^3$
have the same order in discrete valuation ring ${\mathcal O}_{C,(0:0:1)}$
as the minimal order among monomials in $F$.
From this we get $\ord (x) > \ord(y)$, whence
$x^2$ and $y^3$ have the same mimimal order.
This implies that $\delta$ is of length one
which is generated by the class of $x/y$.

(2) Let $C'\subset\bbP^2$ be as in (i) or (ii) above.
Consider the normalization $C\to C'$ of $C'$.
We denote by $\varphi$ the composition $C \to C' \to \bbP^2$.
Set ${\mathcal M} = \varphi^*{\mathcal O}_{\bbP^2}(1)$.
As ${\mathcal M}$ is of degree $5$, we have $\deg(\omega_C\otimes{\mathcal M}^{-1})=3$.
By Riemann-Roch, $h^0({\mathcal M})-h^0(\omega_C\otimes{\mathcal M}^{-1})=1$.
As $h^0({\mathcal M})\ge h^0({\mathcal O}_{\bbP^2}(1))=3$,
we have $h^0(\omega_C\otimes{\mathcal M}^{-1}) \ge 2$.
Hence ${\mathcal M}$ is a special divisor.
By Clifford's theorem 
$h^0({\mathcal M}) \le \deg({\mathcal M})/2+1=7/2$.
Hence $h^0({\mathcal M})=3$ and $h^0(\omega_C\otimes{\mathcal M}^{-1})=2$.
Thus $\omega_C\otimes{\mathcal M}^{-1}$ is a $g_3^1$,
which is base-point-free by Clifford's theorem in the same argument as in 
the second paragraph of the proof of (1).

(3) Let $\nu$ be an automorphism of $C$ over $K$.
By the uniqueness of $\mathcal L$,
this canonically induces an automorphism $\mu$ of
two dimensional projective space $|\omega_C\otimes {\mathcal L}^{-1}|$,
which stabilizes $C'$.
Conversely let $\mu$ be an automorphism of $\bbP^2$ stabilizing $C'$.
Let $\tau: C \to C'$ be the normalization. 
By the universal property of the nomalization,
the composition $\mu\circ \tau: C \to C'$ factors as $\tau\circ \nu$ for a morphism $\nu: C \to C$.
Clearly $\nu$ is an automorphism.
\end{proof}

\begin{prop}
Let $C'$ be of the form above.
The image of the homomorphism in \eqref{StandardExactSquenceCoh}
\[
H^0(\delta) \to H^1(C',{\mathcal O}_{C'}) \simeq H^2(\bbP^2, {\mathcal O}_{\bbP^2}(-5))
\]
is the one-dimensional $K$-vector space generated by
$\displaystyle \frac{1}{xyz^3}$ in the split node case
and in the cusp case and
by $\displaystyle \frac{1}{(x^2-\epsilon y^2)z^3}$ in the non-split node case.
\end{prop}
\begin{proof}
\noindent{\bf \underline{Node case}:}
First consider the split case.
On the affine open subscheme $U_z$ of $C'$ defined by $z\ne 0$,
write
$F = xyz^3 + f$ and choose $u, v$ with $f=u+v$ such that $x^2 | u$ and $y^2 | v$. Put $w:=uv/(xy)^2$. We substitute $1$ for $z$. The equation
\[
\left(\frac{u}{f}\right)^2 - \frac{u}{f} = -\frac{uv}{f^2}
= -\frac{x^2y^2w}{f^2}\equiv - w \modulo\ (F)
\]
says that $\displaystyle\frac{u}{f}$ is integral over $U_z$.

Let us analyze the connecting homomorphism $H^0(\delta) \to H^1(C',{\mathcal O}_{C'})$ of \eqref{StandardExactSquence}.
As a lift of $H^0(\delta)$ we take 
$(\alpha_x,\alpha_y,\alpha_z):=(0,0,u/f)\in \Gamma(U_x,\mathcal O) \times\Gamma(U_y,\mathcal O)\times\Gamma(U_z,\mathcal O)$.
The differences (considered on $U_{xy}, U_{xz}, U_{yz}$) of $\alpha_x,\alpha_y,\alpha_z$ give a cocycle, which defines an element of $H^1(C',{\mathcal O}_{C'})$:
on $U_{xy}$ we have $\beta_{xy}:=\alpha_x-\alpha_y=0$ and
on $U_{xz}$, setting $v'=v/y$ we have
\[
\beta_{xz}:=\alpha_x - \alpha_z = - \frac{u}{f} = -1 + \frac{v}{f}
\equiv -1 - \frac{v}{xyz^3}= -1 - \frac{v'}{xz^3}.
\]
and on $U_{yz}$, setting $u'=u/x$ we have
\[
\beta_{yz}:=\alpha_y-\alpha_z = - \frac{u}{f} \equiv \frac{u}{xyz^3} = \frac{u'}{yz^3}.
\]

Recall the map $H^1(C',{\mathcal O}_{C'}) \to H^2(\bbP^2, {\mathcal O}_{\bbP^2}(-5))$
is obtained as the connecting homomorphism of
\begin{equation}
\begin{CD}
0 @>>> {\mathcal O}_{\bbP^2}(-5) @>F\times>> {\mathcal O}_{\bbP^2} @>>> {\mathcal O}_{C'} @>>> 0.
\end{CD}
\end{equation}
The image of $H^0(\delta) \to H^1(C',{\mathcal O}_{C'}) \to H^2(\bbP^2, {\mathcal O}_{\bbP^2}(-5))$
is obtained by dividing
\[
\beta_{xy}-\beta_{xz}+\beta_{yz}=
0-\left(-1-\frac{v'}{xz^3}\right) + \frac{u'}{yz^3}   =  \frac{F}{xyz^3}.
\]
by $F$.

In the non-split case, use $X=x-\sqrt{\epsilon}y$ and $Y=x+\sqrt{\epsilon}y$.
Then 
over $E:=K[\sqrt{\epsilon}]$
with respect to $X,Y,z$ one can determine as above
the image of the map $H^0(\delta) \to H^1(C',{\mathcal O}_{C'}) \to H^2(\bbP^2, {\mathcal O}_{\bbP^2}(-5))$. It is generated by
\[
\frac{1}{XYz^3} = \frac{1}{(x^2-\epsilon y^2) z^3}.
\]
As the map is defined over $K$ and the compatibility with the base-change
with respect to $E/K$ holds clearly, we have the desired result.
\bigskip

\noindent{\bf \underline{Cusp case}:}
Write $F=x^2g + y(u+v)$ where 
$x^2g$ is the part which does not contain $y$, and
$u$ and $v$ are polynomials chosen so that $x|u$ and $y|v$.
Put $u'=u/x$ and $v'=v/y$.
We claim that $xg/(yz^3)$ is integral over $U_z$.
Indeed, substituting one for $z$, we get
\[
\left(\frac{xg}{y}\right)^2 + u'\left(\frac{xg}{y}\right)+v'g
\equiv -\frac{y(u+v) g}{y^2}+\frac{ug}{y} + \frac{vg}{y}=0.
\]
In the same way, we get the desired result
replacing $u/f$ by $xg/(yz^3)$ in the argument in the node case above.
Specifically we put
$\alpha_x=0$, $\alpha_y=0$ and $\alpha_z= xg/(yz^3)$ over $U_x$, $U_y$ and $U_z$ respectively, and get
$\beta_{xy}=0$ over $U_{xy}$,
\[
\beta_{xz} = - \frac{xg}{yz^3}= - \frac{x^2g}{xyz^3}\equiv \frac{u+v}{xz^3}
\]
over $U_{xz}$ and $\beta_{yz}=-xg/(yz^3)$ over $U_{yz}$.
Then the image of $H^0(\delta) \to H^1(C',{\mathcal O}_{C'}) \to H^2(\bbP^2, {\mathcal O}_{\bbP^2}(-5))$
is obtained by dividing 
\[
\beta_{xy}-\beta_{xz}+\beta_{yz} = - \frac{u+v}{xz^3}-\frac{xg}{yz^3}
=- \frac{F}{xyz^3}
\]
by $F$.
\end{proof}

\subsection{Hasse-Witt matrices}
Let $C$ be a trigonal curve of genus $5$.
Let us give a method to obtain a Hasse-Witt matrix of $C$,
which is a matrix representation of the map
${\mathcal F}_C^*: H^1(C,{\mathcal O}_C) \to H^1(C,{\mathcal O}_C)$
induced by the absolute frobenius ${\mathcal F}_C$ on $C$.
Let $C'$ be the  quintic plane curve associated to $C$, obtained in Lemma \ref{CharacterizationTrigonal}.
Here is a basic commutative diagram
\begin{equation*}
\xymatrix{
H^0(\delta)\ar[r]\ar[dd] & H^2 ( \mathbf{P}^2, {\mathcal O}_{\bbP^2}(-5)))\ar[d]_{{\mathcal F}_{\bbP^2}^{\ast}} &&H^1 \left( C', {\cal O}_{C'} \right)\ar[ll]_{\simeq}\ar[dd]^{\mathcal F_{C'}^{\ast}}\ar[ld]\ar[r] & H^1 \left( C, {\cal O}_C \right)\ar[dd]^{{\mathcal F}_C^{\ast}}\\
                         & H^2 ( \mathbf{P}^2, {\mathcal O}_{\bbP^2}(-5p)))\ar[d]_{F^{p-1}\times} & H^1 \left( {C'}^p, {\cal O}_{{C'}^p} \right)\ar[l]_(0.45){\simeq}\ar[rd] & \\
H^0(\delta)\ar[r] & H^2 ( \mathbf{P}^2, {\mathcal O}_{\bbP^2}(-5))) &&H^1 \left( C', {\cal O}_{C'} \right)\ar[ll]_{\simeq} \ar[r]& H^1 \left( C, {\cal O}_C \right),
%
}
\end{equation*}
where ${C'}^p$ is the closed subscheme of $\bbP^2$ defined by $(F^p)$, and
$\mathcal F_S$ is the absolute Frobenius maps on a scheme $S$ in characteristic $p$. 

\noindent{\bf\underline{Split node case and cusp case}:}\quad
In this case, the space $H^1(C,{\mathcal O}_C)$ is realized as
the quotient of $H^2(\bbP^2,{\mathcal O}_{\bbP^2}(-5))$
by $H^0(\delta)$ the one-dimensional subspace generated by $\displaystyle\frac{1}{xyz^3}$. By the commutative diagram above, we obtain

\begin{prop}\label{HW_ssp_non-split}
Let $C$ be a trigonal curve of genus $5$ which is of split node type or of cusp type.
Let $C'=V(F)$ be the associated quintic in $\bbP^2$ obtained
in {\rm Lemma \ref{CharacterizationTrigonal} (1)}.
Let $c_{l,m}$ {\rm (}$1\le l, m\le 5${\rm )} be the coefficients of the monomials
\[x^{pi_l-i_m}y^{pj_l-j_m}z^{pk_l-k_m}\]
in $F^{p-1}$,
where $(i_1,j_1,k_1)=(3,1,1)$,
$(i_2,j_2,k_2)=(1,3,1)$,
$(i_3,j_3,k_3)=(2,2,1)$,
$(i_4,j_4,k_4)=(2,1,2)$ and
$(i_5,j_5,k_5)=(1,2,2)$.
Then the Hasse-Witt matrix of $C$ is given by $(c_{l,m})$
the square matrix of size $5$.
\end{prop}
Since $C$ is superspecial if and only if its Hasse-Witt matrix is zero, we get
\begin{cor}\label{Criterion_ssp_non-split}
$C$ is superspecial if and only if all the coefficients of the monomials
\[x^{pi-i'}y^{pj-j'}z^{pk-k'}\]
in $F^{p-1}$ are zero, where $(i,j,k)$ and $(i',j',k')$ run through $(3,1,1),(1,3,1),(2,2,1),(2,1,2),(1,2,2)$.
\end{cor}
\noindent{\bf\underline{Non-split node case}:}\quad
The space $H^1(C,{\mathcal O}_C)$ is realized as
the quotient of $H^2(\bbP^2,{\mathcal O}_{\bbP^2}(-5))$
by the one-dimensional subspace generated by $\displaystyle\frac{1}{(x^2-\epsilon y^2)z^3}$.
As a basis of $H^1(C,{\mathcal O}_C)$,
it is natural to use the rational functions of the right hand sides of \eqref{CCM} below.
This basis is obtained by the Zariski cohomology with respect to
the open covering $(U_X, U_Y, U_z)$, where 
$X=x-\sqrt{\epsilon} y$ and $Y=x+\sqrt{\epsilon} y$
and $U_X$ is the part of $X\ne 0$ and so on. Note
\begin{eqnarray}
\frac{1}{X^3Yz}+\frac{1}{XY^3z} &=& \frac{2(x^2+\epsilon y^2)}{(x^2-\epsilon y^2)^3z}, \nonumber \\
\frac{1}{\sqrt{\epsilon}}\left(
\frac{1}{X^3Yz} - \frac{1}{XY^3z} \right) &=& \frac{4xy}{(x^2-\epsilon y^2)^3z},\nonumber \\
\frac{1}{X^2Y^2z} &=& \frac{1}{(x^2-\epsilon y^2)^2z}, \label{CCM}\\
\frac{1}{X^2Yz^2}+\frac{1}{XY^2z^2} &=& \frac{2x}{(x^2-\epsilon y^2)^2z^2},\nonumber \\
\frac{1}{\sqrt{\epsilon}}
\left(\frac{1}{X^2Yz^2}-\frac{1}{XY^2z^2}\right) &=& \frac{2y}{(x^2-\epsilon y^2)^2z^2}.\nonumber 
\end{eqnarray}
The Hasse-Witt matrix $H$ with respect to the basis
consisting of the right hand sides of \eqref{CCM}
is given by $P^{(p)}H'P^{-1}$, where
$H'$ is the Hasse-Witt matrix with respect to
\begin{equation}\label{BasisForH'}
B_1=\frac{1}{X^3Yz},\quad B_2=\frac{1}{XY^3z},\quad B_3=\frac{1}{X^2Y^2z},\quad B_4=\frac{1}{X^2Yz^2},\quad B_5=\frac{1}{XY^2z^2},
\end{equation}
i.e., the $(i,j)$-entry of $H'$ ($1\le i,j\le 5$) is
the coefficient of the monomial $B_iB_j^{-p}$ in $X,Y,z$ 
of
\[
F\left(\frac{X+Y}{2},\frac{X-Y}{-2\sqrt{\epsilon}},z\right)^{p-1},
\]
and $P$ is the coordinate-change matrix corresponding to \eqref{CCM}
\[
P = \begin{pmatrix}
1 & 1 & 0 & 0 & 0\\
\sqrt{\epsilon}^{-1}& -\sqrt{\epsilon}^{-1} & 0 & 0 & 0\\
0 & 0 & 1& 0 &0\\
0 & 0 &  0& 1 & 1 \\
0 & 0 & 0 & \sqrt{\epsilon}^{-1}& -\sqrt{\epsilon}^{-1}\\
\end{pmatrix}
\]
and $P^{(p)}$ is obtained by taking the $p$-th power of each entry of $P$.


%% file: section3.tex
\newcommand{\bfb}{{\mbox{\boldmath $b$}}}

\section{Reduction}\label{SectionReduction}
Let $p$ be a rational prime. Assume $p\ge 5$.
Let $K$ be the finite field $\F_q$ consisting of $q=p^a$ elements.
Let $\zeta$ be a primitive element of $K^\times$.

\subsection{Split node case}
\begin{prop}\label{ReductionSplitNode}
Any trigonal curve over $K$ of split node type
has a quintic model in $\bbP^2$ of the form
\begin{enumerate}
\item[\rm (1)] for $a_i\in K$,
\begin{eqnarray}
F &=& xyz^3 + (x^3 + b_1 y^3) z^2
+ (a_1 x^4 + a_2 x^3 y + a_3 x^2 y^2 + a_4 x y^3 + a_5 y^4) z\nonumber\\
&& + a_6 x^5 + a_7 x^4 y + a_8 x^3 y^2 + a_9 x^2 y^3 + a_{10} x y^4 + a_{11} y^5, \label{SplitNodeReducedEq1}
\end{eqnarray}
where $b_1\in \{0, 1\}$ if $q \equiv 2 \ \modulo\ 3$ and
$b_1\in \{0, 1,\zeta\}$ if $q \equiv 1 \ \modulo\ 3$.
\item[\rm (2)] for $(c_1,c_2) = (0,0), (1,0), (0,1), (1,1), (1,\zeta)$ and for $a_i\in K$,
\begin{eqnarray}
F &=& xyz^3 + (c_1 x^4 + c_2 x^3 y + a_3 x^2 y^2 + a_4 x y^3 + a_5 y^4) z\nonumber\\
&& + a_6 x^5 + a_7 x^4 y + a_8 x^3 y^2 + a_9 x^2 y^3 + a_{10} x y^4 + a_{11} y^5. \label{SplitNodeReducedEq2}
\end{eqnarray}
\end{enumerate}
\end{prop}
\begin{proof}
Considering $z \to z + ax + by$, we can transform the quintic to a quintic
whose coefficients of $x^2yz^2$ and $xy^2z^2$ are zero.
One may write
\begin{eqnarray*}
F &=& xyz^3 + (b_0 x^3 + b_1 y^3) z^2
+ (a_1 x^4 + a_2 x^3 y + a_3 x^2 y^2 + a_4 x y^3 + a_5 y^4) z\\
&& + a_6 x^5 + a_7 x^4 y + a_8 x^3 y^2 + a_9 x^2 y^3 + a_{10} x y^4 + a_{11} y^5.
\end{eqnarray*}

If $(b_0,b_1)\ne (0,0)$, then considering
the exchange of $x$ and $y$
we may assume $b_0\ne 0$.
Consider $(x,y) \to (\alpha x,\beta y)$ and the multiplication by $(\alpha\beta)^{-1}$,
the coefficients $b_0$ and $b_1$ are transformed into $b_0\alpha^2\beta^{-1}$ and $b_1\alpha^{-1}\beta^2$ respectively.
Set $\beta := b_0\alpha^2$; then $b_0$ and $b_1$
become $1$ and $b_0^2b_1\alpha^3$ respectively.
Thus we may assume that $b_0 = 1$ and
$b_1$ is $0$ or a representative of
an element of $K^\times/(K^\times)^3$, i.e.,
$(b_0,b_1) = (1,0), (1,1), (1,\zeta), (1,\zeta^2)$.
Considering the exchange of $x$ and $y$ again, one may reduce to
$(b_0,b_1) = (1,0), (1,1), (1,\zeta)$.

If $(b_0,b_1) = (0,0)$, then there exist elements $\alpha, \beta$ of $K^\times$ such that the transformation $(x,y)\mapsto (\alpha x, \beta y)$
and the multiplication by $(\alpha\beta)^{-1}$ to the whole
make $F$ the form of (2).
\end{proof}

\subsection{Non-split node case}
Let $\epsilon$ be an element of $K^\times \smallsetminus (K^\times)^2$.


\begin{prop}\label{ReductionNonSplitNode}
Any trigonal curve over $K$ of non-split node type
has a quintic model in $\bbP^2$ of the form
\begin{enumerate}
\item[\rm (1)] 
for $a_i\in K$,
\begin{eqnarray}
F &=& (x^2-\epsilon y^2)z^3 + (x(x^2+3\epsilon y^2) + b y(3x^2+\epsilon y^2)) z^2\nonumber\\
&&+ (a_1 x^4 + a_2 x^3 y + a_3 x^2 y^2 + a_4 x y^3 + a_5 y^4) z\nonumber\\
&& + a_6 x^5 + a_7 x^4 y + a_8 x^3 y^2 + a_9 x^2 y^3 + a_{10} x y^4 + a_{11} y^5, \label{NonSplitNodeReducedEq1}
\end{eqnarray}
where $b=0$ if $q \not\equiv -1 \modulo 3$ and
otherwise $b$ has three possibilities determined by the condition that
$(1,b)$ is parallel to $(1,0)A$ with representatives $A$
of $\tilde {\rm C}/\tilde {\rm C}^3$
{\rm (}for example $b=0,6,10$ if $q=11${\rm )}, where
\[
\tilde {\rm C} = \left\{\left.
\begin{pmatrix}r & \epsilon s\\ s & r\end{pmatrix}\right| (r,s) \in K^2,\ (r,s)\ne (0,0) \right\}.
\]
\item[\rm (2)]
 for $c = 1, \zeta$ and for $a_i\in K$,
\begin{eqnarray}
F &=& (x^2-\epsilon y^2)z^3 + (c x^4 + a_2 x^3 y + a_3 x^2 y^2 + a_4 x y^3 + a_5 y^4) z\nonumber\\
&& + a_6 x^5 + a_7 x^4 y + a_8 x^3 y^2 + a_9 x^2 y^3 + a_{10} x y^4 + a_{11} y^5. \label{NonSplitNodeReducedEq2}
\end{eqnarray}
\item[\rm (3)]
for $a_i\in K$,
\begin{eqnarray}
F &=& (x^2-\epsilon y^2)z^3 + a_6 x^5 + a_7 x^4 y + a_8 x^3 y^2 + a_9 x^2 y^3 + a_{10} x y^4 + a_{11} y^5.
\label{NonSplitNodeReducedEq3}
\end{eqnarray}
\end{enumerate}
\end{prop}
\begin{proof}
Let $V$ be the $K$-vector space consisting of cubic forms in $x,y$ over $K$.
As seen in \cite[Lemma 4.1.1]{KH16},
the representation $V$ of $\tilde{\rm C}$
defined by $\gamma x = rx+\epsilon s y$ and $\gamma y = s x +r y$
for
\[
\gamma = \begin{pmatrix}r & \epsilon s\\ s & r\end{pmatrix}\in\tilde{\rm C}
\]
is decomposed as $V_1\oplus V_2$, where
$V_1 = \langle x(x^2-\epsilon y^2), y(x^2-\epsilon y^2)\rangle$ and
$V_2 = \langle x(x^2+3\epsilon y^2), y(3x^2 + \epsilon y^2)\rangle$.
We write
\begin{eqnarray}
F &=& (x^2-\epsilon y^2)z^3 + 
\left\{c_1 x(x^2-\epsilon y^2) + c_2 y(x^2-\epsilon y^2) +  
b_1x(x^2+3\epsilon y^2) + b_2 y(3x^2+\epsilon y^2)\right\} z^2\nonumber\\
&&+ (a_1 x^4 + a_2 x^3 y + a_3 x^2 y^2 + a_4 x y^3 + a_5 y^4) z\nonumber\\
&& + a_6 x^5 + a_7 x^4 y + a_8 x^3 y^2 + a_9 x^2 y^3 + a_{10} x y^4 + a_{11} y^5.
\end{eqnarray}
Considering the transformation sending $z$ to $z - (c_1/3) x - (c_2/3) y$,
we may eliminate the terms of $x(x^2-\epsilon y^2)z^2$
and $y(x^2-\epsilon y^2)z^2$ from $F$, i.e., we may assume $(c_1,c_2)=(0,0)$.

Put $e_1:=x(x^2+3\epsilon y^2)$ and $e_2:=y(3x^2+\epsilon y^2)$.
It is straightforward to see that
the action of $\tilde{\rm C}$ on $V_2$ is described as
$\bfb \to \det(\gamma)^{-1}\gamma^3\bfb$,
with respect to ${\bfb}=\begin{pmatrix}b_1\\b_2\end{pmatrix}$ for $v=b_1e_1+b_2e_2\in V_2$.
Put $E=K[\sqrt{\epsilon}]$.
Use an isomorphism $E^\times \simeq \tilde{\rm C}$ sending
$\alpha=r+s\sqrt{\epsilon}$ ($r,s\in K$) to $\gamma=\begin{pmatrix}r & \epsilon s\\ s & r\end{pmatrix}$,
where $\det(\gamma)$ is equal to the norm $N(\alpha)=\alpha^{q+1}$ of $\alpha$.
Considering the scalar multiplication by $K^\times$ on $V_2$,
the set of the orbits of $K^\times \times \tilde{\rm C}$ in $V_2\smallsetminus\{0\}$ is bijective to
\[
E^\times/(K^\times (E^\times)^3)
= \begin{cases}
1 & q+1 \not\equiv 0 \modulo 3,\\
E^\times/(E^\times)^3 & q+1 \equiv 0 \modulo 3,
\end{cases}
\]
where the equality follows from $K^\times=N(E^\times)=(E^{\times})^{q+1}$.
Thus if $(b_1,b_2)\ne (0,0)$ is not zero, then we have the reduction of the form (1).

Let us consider the case of $(b_1,b_2) = (0,0)$:
 \begin{eqnarray}
F &=& (x^2-\epsilon y^2)z^3 + (a_1 x^4 + a_2 x^3 y + a_3 x^2 y^2 + a_4 x y^3 + a_5 y^4) z\nonumber\\
&& + a_6 x^5 + a_7 x^4 y + a_8 x^3 y^2 + a_9 x^2 y^3 + a_{10} x y^4 + a_{11} y^5.
\end{eqnarray}
Note that an element 
$\begin{pmatrix}r & \epsilon s\\ s & r\end{pmatrix}$ of $\tilde{\rm C}$
transforms $F$ to a quintic with $x^4$-coefficient
$r^4 a_1 + r^3s a_2 + r^2 s^2 a_3 + r s^3 a_4 + s^4 a_5$.
If either of $a_1,a_2,a_3,a_4,a_5$ is not zero, then
$r^4 a_1 + r^3s a_2 + r^2 s^2 a_3 + r s^3 a_4 + s^4 a_5$
is not zero for some $r,s\in K^\times=\F_q^\times$, since $q\ge 5$.
Thus, it is enough to consider the case of $a_1\ne 0$ and the case of
$a_1 = a_2 = a_3 = a_4 = a_5 = 0$.
The quintic in the latter case is of the form (3). In the former case,
considering a transformation $(x,y,z) \mapsto (\alpha x, \alpha y,\beta z)$
($\alpha,\beta \in K^\times$)
and the multiplication by $(\alpha^2\beta^3)^{-1}$ to the whole,
we can make $a_1$ either of $1,\zeta$, i.e., we have the reduction of the form (2).
\end{proof}

\subsection{Cusp case}

\begin{prop}\label{ReductionCusp}
Any trigonal curve over $K$ of cusp type has a quintic model in $\bbP^2$ of the form
\begin{eqnarray}
F &=& x^2 z^3 + a_1 y^3 z^2 + 
(a_2 x^4 + a_3 x^3 y + a_4 x^2 y^2 + b_1 x y^3 + a_5 y^4) z \nonumber\\
&& + a_6 x^5 + a_7 x^4 y + a_8 x^3 y^2 + a_9 x^2 y^3 + b_2 x y^4 + a_{10} y^5 \label{CuspReducedEq}
\end{eqnarray}
for $a_i\in K$ $(i=1,\ldots, 10)$ with $a_1\ne 0$,
where $b_1 \in \{0,1\}$ and $b_2 \in \{0,1\}$.
\end{prop}
\begin{proof}
Since the coefficient of $y^3z^2$ is not zero,
considering $y \to y + cx$ ($c\in K$), we can transform the quintic to a quintic
whose coefficient of $xy^2z^2$ is zero.
Next considering $z \to z + ax + by$, we can fransform the quintic to a quintic
whose coefficients of $x^3z^2$ and $x^2yz^2$ are zero.
Thus we may write $F$ as
\begin{eqnarray*}
F &=& x^2 z^3 + a_1 y^3 z^2 + 
(a_2 x^4 + a_3 x^3 y + a_4 x^2 y^2 + b_1 x y^3 + a_5 y^4) z\\
&& + a_6 x^5 + a_7 x^4 y + a_8 x^3 y^2 + a_9 x^2 y^3 + b_2 x y^4 + a_{10} y^5.
\end{eqnarray*}
Consider $(x,y) \to (\alpha x,\beta y)$ and the multiplication by $\alpha^{-2}$ to the whole,
the coefficients $b_1$ and $b_2$ are transformed into
$\alpha^{-1}\beta^3b_1$ and $\alpha^{-1}\beta^4b_2$ respectively.
If $b_1\ne 0$, then setting $\alpha := \beta^3b_1$; then $b_1$ and $b_2$
become $1$ and $\beta b_1^{-1}b_2$ respectively.
Hence we may assume that $b_1=1$ and $b_2\in\{0,1\}$ by choosing $\beta$ appropriately.
If $b_1 = 0$, then 
we may assume that $b_2\in\{0,1\}$,
by choosing $\alpha$ appropriately. 
\end{proof}

%% file: section4.tex
\section{Main results}\label{sec:main_results}

In this section, we state our main results and prove them.
We put computational parts of the proofs together in Subsection \ref{subsec:comp_result}.
We choose a primitive element $\zeta^{(q)}$ of $\F_q$ for each $q=49$, $11$ and $13$.
Specifically, we set $\zeta^{(49)}:= - 3 - \sqrt{6}$, $\zeta^{(11)}:=2$ and $\zeta^{(13)}:=2$.

\subsection{Superspecial curves}\label{subsec:sscurves}

Our main theorems show the (non-)existence of superspecial trigonal curves of genus $5$ over $\mathbb{F}_q$ for $q=49$, $11$ and $13$.
Specifically, for $q=49$ and $q=13$, there does not exist such a curve, but for $q=11$, there exist precisely $4$ (resp.\ $1$) superspecial trigonal curves of genus $5$ over $\mathbb{F}_{11}$, up to isomorphism over $\mathbb{F}_{11}$ (resp.\ $\overline{\mathbb{F}_{11}}$).

\begin{theor}\label{MainTheorem2}
There is no superspecial trigonal curve of genus $5$ in characteristic $7$.
\end{theor}
\begin{proof}
It suffices to show that there is no superspecial trigonal curve of genus $5$ over $\mathbb{F}_{49}$.
Indeed, a superepecial trigonal curve of genus $5$ in characteristic $7$ descends to a maximal trigonal curve of genus $5$ over $\mathbb{F}_{49}$, which is also superspecial.
Let $C$ be a trigonal curve of genus $5$ over $K=\mathbb{F}_{49}$.
By Lemma \ref{CharacterizationTrigonal}, the trigonal curve $C$ is given as the desingularization of a quintic $C^{\prime}=V(F)$ in $\mathbf{P}^2$ having a node or a cusp which is the unique singular point of $C^{\prime}$.
Here $F$ is a quintic form in $\mathbb{F}_{49} [x,y,z]$.
By Propositions \ref{ReductionSplitNode} -- \ref{ReductionCusp}, we may also assume that $F$ is a quintic form in either of the following five cases:
\begin{enumerate}
\item[1.] Split node case (1):
	\begin{eqnarray*}
	F &=& x y z^3 + (x^3 + b_1 y^3) z^2
	+ (a_1 x^4 + a_2 x^3 y + a_3 x^2 y^2 + a_4 x y^3 + a_5 y^4) z\\
	& & + a_6 x^5 + a_7 x^4 y + a_8 x^3 y^2 + a_9 x^2 y^3 + a_{10} x y^4 + a_{11} y^5
	\end{eqnarray*}
	for $a_i \in \mathbb{F}_{49}$ with $1 \leq i \leq 11$ and for $b_1 \in \{ 0, 1,\zeta^{(49)} \}$.
\item[2.] Split node case (2): 
	\begin{eqnarray*}
	F &=& x y z^3 + (c_1 x^4 + c_2 x^3 y + a_3 x^2 y^2 + a_4 x y^3 + a_5 y^4) z\\
	& & + a_6 x^5 + a_7 x^4 y + a_8 x^3 y^2 + a_9 x^2 y^3 + a_{10} x y^4 + a_{11} y^5
	\end{eqnarray*}
	for $(c_1,c_2) = (0,0), (1,0), (0,1), (1,1), (1,\zeta^{(49)})$ and for $a_i \in \mathbb{F}_{49}$ with $3 \leq i \leq 11$.
\item[3.] Non-split node case (1):
	\begin{eqnarray*}
	F &=& (x^2-\epsilon y^2)z^3 + (x(x^2+3\epsilon y^2) + 0 \cdot y(3 x^2+\epsilon y^2)) z^2\\
	& &+ (a_1 x^4 + a_2 x^3 y + a_3 x^2 y^2 + a_4 x y^3 + a_5 y^4) z\\
	& & + a_6 x^5 + a_7 x^4 y + a_8 x^3 y^2 + a_9 x^2 y^3 + a_{10} x y^4 + a_{11} y^5
	\end{eqnarray*}
	for $a_i \in \mathbb{F}_{49}$ with $1 \leq i \leq 11$, where $\epsilon$ is an element of $\mathbb{F}_{49}^{\times} \smallsetminus (\mathbb{F}_{49}^{\times})^2$.
\item[4.] Non-split node case (2):
	\begin{eqnarray*}
	F & = & (x^2-\epsilon y^2)z^3 + (c x^4 + a_2 x^3 y + a_3 x^2 y^2 + a_4 x y^3 + a_5 y^4) z\\
	& & + a_6 x^5 + a_7 x^4 y + a_8 x^3 y^2 + a_9 x^2 y^3 + a_{10} x y^4 + a_{11} y^5
	\end{eqnarray*}
	for $c \in \{1,\zeta^{(49)} \}$ and for $a_i \in \mathbb{F}_{49}$ with $2 \leq i \leq 11$.
	\item[5.] Non-split node case (3):
	\begin{eqnarray*}
	F & = & (x^2-\epsilon y^2)z^3 + a_6 x^5 + a_7 x^4 y + a_8 x^3 y^2 + a_9 x^2 y^3 + a_{10} x y^4 + a_{11} y^5
	\end{eqnarray*}
	for $a_i \in \mathbb{F}_{49}$ with $6 \leq i \leq 11$.
\item[6.] Cusp case:
	\begin{eqnarray*}
	F &=& x^2 z^3 + a_1 y^3 z^2 + 
	(a_2 x^4 + a_3 x^3 y + a_4 x^2 y^2 + b_1 x y^3 + a_5 y^4) z\\
	&& + a_6 x^5 + a_7 x^4 y + a_8 x^3 y^2 + a_9 x^2 y^3 + b_2 x y^4 + a_{10} y^5
	\end{eqnarray*}
	for $a_i \in \mathbb{F}_{49}$ with $1 \leq i \leq 10$ and $a_1 \neq 0$, where $b_1 \in \{ 0, 1 \}$ and $b_2 \in \{ 0, 1 \}$.
\end{enumerate}
It follows from Propositions \ref{prop:Case1q7} -- \ref{prop:Case5q7} in Subsection \ref{subsec:comp_result} that for each of the above five cases, there does not exist any quintic form $F$ of the form stated in the case such that the desingularization of $V(F)$ is a superspecial trigonal curve of genus $5$.
\end{proof}

\begin{theor}\label{MainTheorem}
Any superspecial trigonal curve of genus $5$ over $\F_{11}$ is $\F_{11}$-isomorphic to the desingularization of
\begin{equation}
x y z^3 + a_1 x^5 + a_2 y^5  = 0 \label{sscurve_F11}
\end{equation}
in $\bbP^2$, where $a_1, a_2 \in \mathbb{F}_{11}^\times$, or the desingularization of
\begin{equation}
 (x^2 - \epsilon y^2) z^3 + a x^5 + b x^4 y + (9 a) x^3 y^2 + 4 b x^2 y^3 + ( 9 a ) x y^4 + 3 b y^5 =0 \label{sscurve_F11_2}
\end{equation}
in $\bbP^2$, where $\epsilon \in \mathbb{F}_{11}^\times \smallsetminus (\mathbb{F}_{11}^\times)^2$ and $(a, b) \in (\mathbb{F}_{11})^{\oplus 2} \smallsetminus \{ (0,0 )\}$.
\end{theor}

\begin{proof}
Let $C$ be a trigonal curve of genus $5$ over $K=\mathbb{F}_{11}$.
By the same argument as in the proof of Theorem \ref{MainTheorem2}, the curve $C$ is assumed to be given as the desingularization of the quintic in $\mathbf{P}^2$ defined by $F \in \mathbb{F}_{11}[x,y,z]$, where $F$ is a quintic form in either of the following five cases:
\begin{enumerate}
\item[1.] Split node case (1):
	\begin{eqnarray*}
	F &=& x y z^3 + (x^3 + b_1 y^3) z^2
	+ (a_1 x^4 + a_2 x^3 y + a_3 x^2 y^2 + a_4 x y^3 + a_5 y^4) z\\
	& & + a_6 x^5 + a_7 x^4 y + a_8 x^3 y^2 + a_9 x^2 y^3 + a_{10} x y^4 + a_{11} y^5
	\end{eqnarray*}
	for $a_i \in \mathbb{F}_{11}$ with $1 \leq i \leq 11$ and for $b_1 \in \{ 0, 1 \}$.
\item[2.] Split node case (2): 
	\begin{eqnarray*}
	F &=& x y z^3 + (c_1 x^4 + c_2 x^3 y + a_3 x^2 y^2 + a_4 x y^3 + a_5 y^4) z\\
	& & + a_6 x^5 + a_7 x^4 y + a_8 x^3 y^2 + a_9 x^2 y^3 + a_{10} x y^4 + a_{11} y^5
	\end{eqnarray*}
	for $(c_1,c_2) = (0,0), (1,0), (0,1), (1,1), (1,\zeta^{(11)})$ and for $a_i \in \mathbb{F}_{11}$ with $3 \leq i \leq 11$.
\item[3.] Non-split node case (1):
	\begin{eqnarray*}
	F &=& (x^2-\epsilon y^2)z^3 + (x(x^2+3\epsilon y^2) + b y(3 x^2+\epsilon y^2)) z^2\\
	& &+ (a_1 x^4 + a_2 x^3 y + a_3 x^2 y^2 + a_4 x y^3 + a_5 y^4) z\\
	& & + a_6 x^5 + a_7 x^4 y + a_8 x^3 y^2 + a_9 x^2 y^3 + a_{10} x y^4 + a_{11} y^5
	\end{eqnarray*}
	for $a_i \in \mathbb{F}_{11}$ with $1 \leq i \leq 11$ and for $b \in \{ 0, 6, 10 \}$, where $\epsilon$ is an element of $\mathbb{F}_{11}^{\times} \smallsetminus (\mathbb{F}_{11}^{\times})^2$.
\item[4.] Non-split node case (2):
	\begin{eqnarray*}
	F & = & (x^2-\epsilon y^2)z^3 + (c x^4 + a_2 x^3 y + a_3 x^2 y^2 + a_4 x y^3 + a_5 y^4) z\\
	& & + a_6 x^5 + a_7 x^4 y + a_8 x^3 y^2 + a_9 x^2 y^3 + a_{10} x y^4 + a_{11} y^5
	\end{eqnarray*}
	for $c \in \{1,\zeta^{(11)}\}$ and for $a_i \in \mathbb{F}_{11}$ with $2 \leq i \leq 11$.
	\item[5.] Non-split node case (3):
	\begin{eqnarray*}
	F & = & (x^2-\epsilon y^2)z^3 + a_6 x^5 + a_7 x^4 y + a_8 x^3 y^2 + a_9 x^2 y^3 + a_{10} x y^4 + a_{11} y^5
	\end{eqnarray*}
	for $a_i \in \mathbb{F}_{11}$ with $6 \leq i \leq 11$.
\item[6.] Cusp case:
	\begin{eqnarray*}
	F &=& x^2 z^3 + a_1 y^3 z^2 + 
	(a_2 x^4 + a_3 x^3 y + a_4 x^2 y^2 + b_1 x y^3 + a_5 y^4) z\\
	&& + a_6 x^5 + a_7 x^4 y + a_8 x^3 y^2 + a_9 x^2 y^3 + b_2 x y^4 + a_{10} y^5
	\end{eqnarray*}
	for $a_i \in \mathbb{F}_{11}$ with $1 \leq i \leq 10$ and $a_1 \neq 0$, where $b_1 \in \{ 0, 1 \}$ and $b_2 \in \{ 0, 1 \}$.
\end{enumerate}
It follows from Propositions \ref{prop:Case1q11}, \ref{prop:Case3q11} -- \ref{prop:Case5q11} in Subsection \ref{subsec:comp_result} that for each of the cases 1, 3, 4 and 6, there does not exist any quintic form $F$ of the form stated in the case such that the desingularization of $V(F)$ is a superspecial trigonal curve of genus $5$.
By Propositions \ref{prop:Case2q11} and \ref{prop:Case4q11} in Subsection \ref{subsec:comp_result}, the claim holds.
\end{proof}

\begin{theor}\label{MainTheorem3}
There is no superspecial trigonal curve of genus $5$ over $\mathbb{F}_{13}$.
\end{theor}
\begin{proof}
Similarly to the proof of Theorem \ref{MainTheorem2}, the claim follows from Propositions \ref{prop:Case1q13} and \ref{prop:Case2q13} -- \ref{prop:Case5q13} in Subsection \ref{subsec:comp_result}.
\end{proof}

\begin{rem}
As we mentioned in Section \ref{sec:Intro}, it is theoretically proved that there does not exist any superspecial trigonal curve of genus $5$ over $\mathbb{F}_{25}$.
We have also checked this in a way similar to the (computational) proofs of Theorems \ref{MainTheorem2} -- \ref{MainTheorem3}, but let us omit to write details in this paper.
\end{rem}

\subsection{Algorithms for enumerating superspecial trigonal curves of genus $5$}
\label{subsec:algorithm}
This subsection presents algorithms for enumerating superspecial trigonal curves of genus $5$ over $K=\mathbb{F}_q$.
Let $C$ be a trigonal curve of genus $5$ over $\mathbb{F}_q$.
As we showed in Subsection \ref{TrigonalCurvesGenus5}, the curve $C$ is given as the desingularization of a quintic $C^{\prime}$ in $\mathbf{P}^2$ having a node or a cusp which is the unique singular point of $C^{\prime}$.
Let $F$ be a quintic form in $\mathbb{F}_q[x,y,z]$ defining $C^{\prime}$.
The quintic form $F$ is written as
\begin{eqnarray}
F = \sum_{i=1}^t a_i p_i + \sum_{j=1}^u b_j q_j , \label{eq:quintic}
\end{eqnarray}
where $p_i$'s and $q_j$'s are monomials of degree $5$ or quintic forms with a few monomials, and where $a_i$'s and $b_j$'s are elements in $\mathbb{F}_q$.
Assume that each $b_j$, which is the coefficient of $q_j$, takes an element in a small range, e.g., $\{ 0, 1 \}$.
As we showed in Section \ref{SectionReduction}, the forms of $F$ have three cases: (Split node case), (Non-split node case) and (Cusp case).

Thus, for enumerating all superspecial trigonal curves of genus $5$ over $K=\mathbb{F}_q$, it suffices to enumerate irreducible quintic forms $F \in \mathbb{F}_q [x,y,z]$ of the form \eqref{eq:quintic} such that
\begin{itemize}
\item $V(F)$ has a node or a cusp which is its unique singular point, and
\item the desingularization of $V (F)$ is superspecial and it has (geometric) genus $5$,
\end{itemize}
in each of the three cases.

\subsubsection{Algorithm for (Split node case) and (Cusp case)}

Here, we give an algorithm for (Split node case) and (Cusp case).
For simplicity, we assume that $F$ is of the form $F = \sum_{i=1}^t a_i p_i$ for some quintic forms $p_i$'s in $\mathbb{F}_q[x,y,z]$.

\paragraph{Enumeration Algorithm 1}
\begin{description}
\item[{\it Input}{\rm :}] (i) A rational prime $p$, (ii) a power $q$ of the prime $p$, and (iii) a set $\{ p_1, \ldots , p_t \}$ of quintic forms in $\mathbb{F}_q [x,y,z]$.
\item[{\it Output}{\rm :}] A list of quintic forms of the form $\sum_{i=1}^t a_i p_i$ for $a_i \in \mathbb{F}_q$ with $1 \leq i \leq t$.
\end{description}
Let $\mathcal{F}:=\emptyset$, and
\begin{enumerate}
\item[{\rm (0)}]
Regard some unknown coefficients in $F = \sum_{i=1}^t a_i p_i$ as indeterminates.
Specifically, choose $1 \leq s_1 \leq t$ and indices $1 \leq k_1 < \cdots < k_{s_1} \leq t$, and then regard $a_{k_1}, \ldots , a_{k_{s_1}}$ as indeterminates.
Assume for simplicity that the first $s_1$ coefficients $a_1, \ldots , a_{s_1}$ with $s_1 \leq t$ are indeterminates here, i.e., we take $(k_1, \ldots , k_{s_1})$ to be $(1, \ldots , s_1)$.
The remaining part $(a_{s_1+1}, \ldots , a_{t})$ runs through a subset $\mathcal{A}_1 \subset (\mathbb{F}_q)^{\oplus t-s_1}$.
\end{enumerate}
For each $(c_{s_1+1}, \ldots , c_{t}) \in \mathcal{A}_1$, proceed with the following three steps:
\begin{enumerate}
	\item[{\rm (1)}] Substitute respectively $c_{s_1+1}, \ldots , c_{t}$ into $a_{s_1+1}, \ldots , a_{t}$ in $F$, and compute $h:=F^{p-1}$ over $\mathbb{F}_{q} [a_1, \ldots, a_{s_1}] [x, y, z]$.
	\item[{\rm (2)}] Regard some unknown coefficients among $a_1, \ldots , a_{s_1}$ as indeterminates.
	Specifically, choose $1 \leq s_2 \leq s_1$ and indices $1 \leq i_1 < \cdots < i_{s_2} \leq s_1$, and then regard $a_{i_1}, \ldots , a_{i_{s_2}}$ as indeterminates.
Assume for simplicity that the first $s_2$ coefficients $a_{1}, \ldots , a_{s_2}$ with $s_2 \leq s_1$ are indeterminates here, i.e., we take $(i_1, \ldots , i_{s_2})$ to be $(1, \ldots , s_2)$.
The remaining part $(a_{s_2+1}, \ldots , a_{s_1})$ runs through a subset $\mathcal{A}_2 \subset (\mathbb{F}_q)^{\oplus s_1-s_2}$.
	\item[{\rm (3)}] Let $\mathcal{S} \subset \mathbb{F}_{q} [a_1, \ldots, a_{s_1}]$ be the set of the coefficients of the $25$ monomials in $h=F^{p-1}$, given in Corollary \ref{Criterion_ssp_non-split}.
We proceed with the following three steps for each $(c_{s_2+1}, \ldots , c_{s_1} ) \in \mathcal{A}_2$: 
	\begin{enumerate}
		\item[{\rm (a)}] For each $f \in \mathcal{S}$, substitute respectively $c_{s_2+1}, \ldots , c_{s_1}$ into $a_{s_2+1}, \ldots , a_{s_1}$ in $f$.
		Put 
		\[
		\mathcal{S}^{\prime}:= \{ f (a_1, \ldots , a_{s_2}, c_{s_2+1}, \ldots , c_{s_1}) : f \in \mathcal{S} \} \cup \{ a_{s_2+1} - c_{s_2+1}, \ldots a_{s_1} - c_{s_1} \} .
		\]
		\item[{\rm (b)}] Solve the multivariate system $g  = 0$ for all $g \in \mathcal{S}^{\prime}$ over $\mathbb{F}_{q}$ with known algorithms via the Gr\"{o}bner basis computation (e.g., an algorithm given in \cite[Subsection 2.3]{KHS17}).
		\item[{\rm (c)}] For each root of the above system, substitute it into corresponding unknown coefficients in $F$, and decide whether $C^{\prime} = V ( F )$ has just one singular point or not. 
		If $C^{\prime}$ has just one singular point, replace $\mathcal{F}$ by $\mathcal{F} \cup \{ F \}$.
	\end{enumerate}
\end{enumerate}
Return $\mathcal{F}$.

\begin{prop}
With notation as above, {\rm Enumeration Algorithm 1} outputs the list of all the quintic forms of the form $F=\sum_{i=1}^t a_i p_i$ such that if $F$ is irreducible over $\overline{\mathbb{F}_q}$, the desingularizations of $V(F) \subset \mathbf{P}^2$ are superspecial trigonal curves of genus $5$ over $\mathbb{F}_q$. 
\end{prop}

\begin{proof}
This follows immediately from Corollary \ref{Criterion_ssp_non-split} together with the construction of the algorithm.
\end{proof}

\subsubsection{Algorithm for (Non-split node case)}

We give an algorithm for the non-split node case.
For simplicity, we assume that $F$ is of the form $F = \sum_{i=1}^t a_i p_i$ for some quintic forms $p_i$'s in $\mathbb{F}_q[x,y,z]$.
Let $\epsilon$ be an element in $\mathbb{F}_q^{\times} \smallsetminus (\mathbb{F}_q^{\times})^2$.

\paragraph{Enumeration Algorithm 2}
\begin{description}
\item[{\it Input}{\rm :}] (i) A rational prime $p$, (ii) a power $q$ of the prime $p$, and (iii) a set $\{ p_1, \ldots , p_t \}$ of quintic forms in $\mathbb{F}_q [x,y,z]$.
\item[{\it Output}{\rm :}] A list of quintic forms of the form $\sum_{i=1}^t a_i p_i$ for $a_i \in \mathbb{F}_q$ with $1 \leq i \leq t$.
\end{description}
Construct the quadratic extension field $K^{\prime}:=\mathbb{F}_q [t] / \langle t^2 - \epsilon \rangle \cong \mathbb{F}_{q^2}$, where we interpret $t = \sqrt{\epsilon}$.
Compute
\[
P = \begin{pmatrix}
1 & 1 & 0 & 0 & 0\\
\sqrt{\epsilon}^{-1}& -\sqrt{\epsilon}^{-1} & 0 & 0 & 0\\
0 & 0 & 1& 0 &0\\
0 & 0 &  0& 1 & 1 \\
0 & 0 & 0 & \sqrt{\epsilon}^{-1}& -\sqrt{\epsilon}^{-1}\\
\end{pmatrix}
\]
and $P^{-1}$.
Compute $P^{(p)}$, which is the matrix obtained by taking the $p$-th power of each entry of $P$.
Let $\mathcal{F}:=\emptyset$, and
\begin{enumerate}
\item[{\rm (0)}]
Regard some unknown coefficients in $F = \sum_{i=1}^t a_i p_i$ as indeterminates.
Specifically, choose $1 \leq s_1 \leq t$ and indices $1 \leq k_1 < \cdots < k_{s_1} \leq t$, and then regard $a_{k_1}, \ldots , a_{k_{s_1}}$ as indeterminates.
Assume for simplicity that the first $s_1$ coefficients $a_1, \ldots , a_{s_1}$ with $s_1 \leq t$ are indeterminates here, i.e., we take $(k_1, \ldots , k_{s_1})$ to be $(1, \ldots , s_1)$.
The remaining part $(a_{s_1+1}, \ldots , a_{t})$ runs through a subset $\mathcal{A}_1 \subset (\mathbb{F}_q)^{\oplus t-s_1}$.
\end{enumerate}
For each $(c_{s_1+1}, \ldots , c_{t}) \in \mathcal{A}_1$, proceed with the following five steps:
\begin{enumerate}
	\item[{\rm (1)}] Substitute respectively $c_{s_1+1}, \ldots , c_{t}$ into $a_{s_1+1}, \ldots , a_{t}$ in $F$, and compute $h:=F^{p-1}$ over $K^{\prime} [a_1, \ldots, a_{s_1}] [x, y, z]$.
	\item[{\rm (2)}] Regard some unknown coefficients among $a_1, \ldots , a_{s_1}$ as indeterminates.
	Specifically, choose $1 \leq s_2 \leq s_1$ and indices $1 \leq i_1 < \cdots < i_{s_2} \leq s_1$, and then regard $a_{i_1}, \ldots , a_{i_{s_2}}$ as indeterminates.
Assume for simplicity that the first $s_2$ coefficients $a_{1}, \ldots , a_{s_2}$ with $s_2 \leq s_1$ are indeterminates here, i.e., we take $(i_1, \ldots , i_{s_2})$ to be $(1, \ldots , s_2)$.
The remaining part $(a_{s_2+1}, \ldots , a_{s_1})$ runs through a subset $\mathcal{A}_2 \subset (\mathbb{F}_q)^{\oplus s_1-s_2}$.
	\item[{\rm (3)}] Substitute $\frac{X+Y}{2}$ and $\frac{X-Y}{-2\sqrt{\epsilon}}$ into $x$ and $y$ in $h=F^{p-1}$.
	Let $h^{\prime}$ denote the polynomial transformed from $h$ by the above substitution.
	\item[{\rm (4)}] Let $H^{\prime}$ be the Hasse-Witt matrix with respect to \eqref{BasisForH'}, i.e., the $(i,j)$-entry of $H^{\prime}$ ($1 \leq i,j \leq 5$) is the coefficient of the monomial $B_i B_j^{-p}$ in $h^{\prime}$.
	Here $B_i$ for $1 \leq i \leq 5$ are given in \eqref{BasisForH'}.
	\item[{\rm (5)}] Compute $H:=P^{(p)} H^{\prime} P^{-1}$.
	Let $\mathcal{S}$ be the set of the $25$ entries of $H$.
	Note that $\mathcal{S} \subset \mathbb{F}_{q} [a_1, \ldots, a_{s_1}]$.
We proceed with the following three steps for each $(c_{s_2+1}, \ldots , c_{s_1} ) \in \mathcal{A}_2$: 
	\begin{enumerate}
		\item[{\rm (a)}] For each $f \in \mathcal{S}$, substitute respectively $c_{s_2+1}, \ldots , c_{s_1}$ into $a_{s_2+1}, \ldots , a_{s_1}$ in $f$.
		Put 
		\[
		\mathcal{S}^{\prime}:= \{ f (a_1, \ldots , a_{s_2}, c_{s_2+1}, \ldots , c_{s_1}) : f \in \mathcal{S} \} \cup \{ a_{s_2+1} - c_{s_2+1}, \ldots a_{s_1} - c_{s_1} \} .
		\]
		\item[{\rm (b)}] Solve the multivariate system $g  = 0$ for all $g \in \mathcal{S}^{\prime}$ over $\mathbb{F}_{q}$ with known algorithms via the Gr\"{o}bner basis computation.
		\item[{\rm (c)}] For each root of the above system, substitute it into corresponding unknown coefficients in $F$, and decide whether $C^{\prime} = V ( F )$ has just one singular point or not. 
		If $C^{\prime}$ has just one singular point, replace $\mathcal{F}$ by $\mathcal{F} \cup \{ F \}$.
	\end{enumerate}
\end{enumerate}
Return $\mathcal{F}$.

\begin{prop}
With notation as above, {\rm Enumeration Algorithm 2} outputs the list of all the quintic forms of the form $F=\sum_{i=1}^t a_i p_i$ such that if $F$ is irreducible over $\overline{\mathbb{F}_q}$, the desingularizations of $V(F) \subset \mathbf{P}^2$ are superspecial trigonal curves of genus $5$ over $\mathbb{F}_q$. 
\end{prop}

\begin{proof}
This follows from Corollary \ref{Criterion_ssp_non-split} together with the construction of the algorithm.
\end{proof}

\subsection{Computational parts of our proofs of the main theorems}\label{subsec:comp_result}
In this subsection, we give computational results for the proofs of our main theorems (Theorems \ref{MainTheorem2}, \ref{MainTheorem} and \ref{MainTheorem3}).
The computational results are obtained by executing Enumeration Algorithms 1 and 2 in Subsection \ref{subsec:algorithm}.
We implemented and executed the algorithms over Magma V2.22-7 \cite{Magma} in its 64-bit version (for details on the implementation, see Subsection \ref{subsec:imple}).

\subsubsection{Split node case (1) with $q = p^2 = 49$}

\begin{prop}\label{prop:Case1q7}
Consider the quintic form
\begin{equation}
\begin{split}
F  = &  x y z^3 + (x^3 + b_1 y^3) z^2 + (a_1 x^4 + a_2 x^3 y + a_3 x^2 y^2 + a_4 x y^3 + a_5 y^4) z \\
& + a_6 x^5 + a_7 x^4 y + a_8 x^3 y^2 + a_9 x^2 y^3 + a_{10} x y^4 + a_{11} y^5,
\end{split}\label{eq:quintic1}
\end{equation}
where $a_i \in \mathbb{F}_{49}$ for $1 \leq i \leq 11$ and $b_1 \in \{ 0, 1 , \zeta^{(49)} \}$.
Then there does not exist any quintic form $F$ of the form \eqref{eq:quintic1} such that the desingularization of $V (F) \subset \mathbf{P}^2$ is a superspecial trigonal curve of genus $5$. 
\end{prop}

\begin{proof}
Put $t = 11$, $u = 3$ and
\begin{eqnarray}
\{ p_1, \ldots , p_t \} & = & \{ x^4 z, x^3 y z, x^2 y^2 z, x y^3 z, y^4 z,  x^5, x^4 y, x^3 y^2, x^2 y^3, x y^4, y^5 \}, \nonumber \\
\{ q_1, \ldots , q_u \} & = & \{ y^3 z^2, x y z^3, x^3 z^2 \}. \nonumber 
\end{eqnarray}
For each $b_1 \in \{ 0, 1, \zeta^{(49)} \}$, execute Enumeration Algorithm 1 given in Subsection \ref{subsec:algorithm}.
We here give an outline of our computation together with our choices of $s_1$, $s_2$, $\{ k_1, \ldots , k_{s_1} \}$, $\{ i_1, \ldots , i_{s_2} \}$, $\mathcal{A}_1$, $\mathcal{A}_2$ and a term ordering in the algorithm.
\begin{enumerate}
\item[{\rm (0)}]
We set $s_1:= 11$, and $( k_1, \ldots , k_{s_1} ):= ( 1, \ldots , 11)$ (we regard the $11$ coefficients $a_1$, $a_2$, $a_3$, $a_4$, $a_5$, $a_6$, $a_7$, $a_8$, $a_9$, $a_{10}$ and $a_{11}$ as indeterminates).
Let $\mathcal{A}_1:=\emptyset$.
\end{enumerate}
We proceed with the following three steps:
\begin{enumerate}
	\item[{\rm (1)}] Compute $F:= \sum_{i=1}^t a_i p_i + \sum_{j=1}^u b_j q_j$ and $h:= ( F )^{p-1}$ over $\mathbb{F}_{49} [a_1, \ldots, a_{11}] [x, y, z]$, where $a_1, \ldots, a_{11}$ are indeterminates, and $(b_2, b_3)=(1,1)$.
	\item[{\rm (2)}] We set $s_2 := 9$, and $(i_1, \ldots , i_{s_2}) := ( 2, 4, 5, 6, 7, 8, 9, 10, 11)$ (we regard the $9$ coefficients $a_2$, $a_4$, $a_5$, $a_6$, $a_7$, $a_8$, $a_9$, $a_{10}$ and $a_{11}$ as indeterminates).
	For the Gr\"{o}bner basis computation in $\mathbb{F}_{49} [ a_1, a_2, a_3, a_4, a_5, a_6, a_7, a_8, a_9, a_{10}, a_{11}]$ below, we adopt the graded reverse lexicographic (grevlex) order with
\[
a_{11} \prec a_{10} \prec a_9 \prec a_8 \prec a_7 \prec a_6 \prec a_5 \prec a_4 \prec a_3 \prec a_2 \prec a_1,
\]
	whereas for that in $\mathbb{F}_{49} [ x, y, z]$, the grevlex order with $z \prec y \prec x$ is adopted.
	Put $\mathcal{A}_2 := ( \mathbb{F}_{49} )^{\oplus 2}$.
	\item[{\rm (3)}] Let $\mathcal{S} \subset \mathbb{F}_{49} [a_1, \ldots, a_{11}]$ be the set of the coefficients of the $25$ monomials in $h$, given in Corollary \ref{Criterion_ssp_non-split}.
We proceed with the following three steps for each $( c_1, c_3 ) \in \mathcal{A}_2= ( \mathbb{F}_{49} )^{\oplus 2}$: 
	\begin{enumerate}
		\item[{\rm (a)}] For each $f \in \mathcal{S}$, substitute respectively $c_1$ and $c_3$ into $a_1$ and $a_3$ in $f$.
		Put 
		\[
		\mathcal{S}^{\prime}:= \{ f (c_1, a_2, c_3, a_4, \ldots , a_{11}) : f \in \mathcal{S} \} \cup \{ a_1 - c_1, a_3 - c_3 \} .
		\]
		\item[{\rm (b)}] Solve the multivariate system $g  = 0$ for all $g \in \mathcal{S}^{\prime}$ over $\mathbb{F}_{49}$ with known algorithms via the Gr\"{o}bner basis computation.
		\item[{\rm (c)}] For each root of the above system, substitute it into corresponding unknown coefficients in $F$, and decide whether $C = V ( F )$ has just one node or not. 
	\end{enumerate}
\end{enumerate}
From the outputs, we have that there does not exist any quintic form $F$ of the form \eqref{eq:quintic1} such that the desingularization of $V (F) \subset \mathbf{P}^2$ is a superspecial trigonal curve of genus $5$. 
\end{proof}

\subsubsection{Split node case (2) with $q = p^2 = 49$}

\begin{prop}\label{prop:Case2q7}
Consider the quintic form
\begin{equation}
\begin{split}
F  = &  x y z^3 + (c_1 x^4 + c_2 x^3 y + a_3 x^2 y^2 + a_4 x y^3 + a_5 y^4) z \\
& + a_6 x^5 + a_7 x^4 y + a_8 x^3 y^2 + a_9 x^2 y^3 + a_{10} x y^4 + a_{11} y^5,
\end{split}\label{eq:quintic2}
\end{equation}
where $a_i \in \mathbb{F}_{49}$ for $3 \leq i \leq 11$ and $(c_1, c_2) \in \{ (0,0), (1,0), (0,1), (1,1), (1,\zeta^{(49)}) \}$.
Then there does not exist any quintic form $F$ of the form \eqref{eq:quintic2} such that the desingularization of $V (F) \subset \mathbf{P}^2$ is a superspecial trigonal curve of genus $5$. 
\end{prop}

\begin{proof}
Put $t = 11$, $u = 1$ and
\begin{eqnarray}
\{ p_1, \ldots , p_t \} & = & \{  x^4 z, x^3 y z, x^2 y^2 z, x y^3 z, y^4 z,  x^5, x^4 y, x^3 y^2, x^2 y^3, x y^4, y^5 \}, \nonumber \\
\{ q_1, \ldots , q_u \} & = & \{ x y z^3 \}. \nonumber 
\end{eqnarray}
Execute Enumeration Algorithm 1 given in Subsection \ref{subsec:algorithm}.
We here give an outline of our computation together with our choices of $s_1$, $s_2$, $\{ k_1, \ldots , k_{s_1} \}$, $\{ i_1, \ldots , i_{s_2} \}$, $\mathcal{A}_1$, $\mathcal{A}_2$ and a term ordering in the algorithm.
\begin{enumerate}
\item[{\rm (0)}]
We set $s_1:= 11$, and $( k_1, \ldots , k_{s_1} ):= ( 1, \ldots , 11)$ (we regard the $11$ coefficients $a_1$, $a_2$, $a_3$, $a_4$, $a_5$, $a_6$, $a_7$, $a_8$, $a_9$, $a_{10}$ and $a_{11}$ as indeterminates).
Let $\mathcal{A}_1:=\emptyset$.
\end{enumerate}
We proceed with the following three steps:
\begin{enumerate}
	\item[{\rm (1)}] Compute $F:= \sum_{i=1}^t a_i p_i + \sum_{j=1}^u b_j q_j$ and $h:= ( F )^{p-1}$ over $\mathbb{F}_{49} [a_1, \ldots, a_{11}] [x, y, z]$, where $a_1, \ldots, a_{11}$ are indeterminates and $b_1=1$. 
	\item[{\rm (2)}] We set $s_2 := 9$, and $(i_1, \ldots , i_{s_2}) := ( 3, 4, 5, 6, 7, 8, 9, 10, 11)$ (we regard the $9$ coefficients $a_3$, $a_4$, $a_5$, $a_6$, $a_7$, $a_8$, $a_9$, $a_{10}$ and $a_{11}$ as indeterminates).
	For the Gr\"{o}bner basis computation in $\mathbb{F}_{49} [ a_1, a_2, a_3, a_4, a_5, a_6, a_7, a_8, a_9, a_{10}, a_{11}]$ in (3), we adopt the graded reverse lexicographic (grevlex) order with
\[
a_{11} \prec a_{10} \prec a_9 \prec a_8 \prec a_7 \prec a_6 \prec a_5 \prec a_4 \prec a_3 \prec a_2 \prec a_1,
\]
	whereas for that in $\mathbb{F}_{49} [ x, y, z]$, the grevlex order with $z \prec y \prec x$ is adopted.
	Put $\mathcal{A}_2 := \{ (0,0), (1,0), (0,1), (1,1), (1,\zeta^{(49)}) \} $.
	\item[{\rm (3)}] Let $\mathcal{S} \subset \mathbb{F}_{49} [a_1, \ldots, a_{11}]$ be the set of the coefficients of the $25$ monomials in $h$, given in Corollary \ref{Criterion_ssp_non-split}.
	For each $( c_1, c_2 ) \in \mathcal{A}_2$, conduct procedures similar to the proof of Proposition \ref{prop:Case1q7}.
\end{enumerate}
From the outputs, we have that there does not exist any quintic form $F$ of the form \eqref{eq:quintic2} such that the desingularization of $V (F) \subset \mathbf{P}^2$ is a superspecial trigonal curve of genus $5$. 
\end{proof}

\subsubsection{Non-split node case (1) with $q = p^2 = 49$}

\begin{prop}\label{prop:Case3q7}
Consider the quintic form
\begin{equation}
\begin{split}
F  = &  (x^2 - \epsilon y^2 ) z^3 + x (x^2 + 3 \epsilon y^2 ) z^2 + 0 \cdot y (3 x^2 + \epsilon y^2 ) z^2 \\
& + (a_1 x^4 + a_2 x^3 y + a_3 x^2 y^2 + a_4 x y^3 + a_5 y^4) z \\
& + a_6 x^5 + a_7 x^4 y + a_8 x^3 y^2 + a_9 x^2 y^3 + a_{10} x y^4 + a_{11} y^5,
\end{split}\label{eq:quintic3}
\end{equation}
where $a_i \in \mathbb{F}_{49}$ for $1 \leq i \leq 11$, and $\epsilon \in \mathbb{F}_{49}^{\times} \smallsetminus (\mathbb{F}_{49}^{\times})^2$.
Then there does not exist any quintic form $F$ of the form \eqref{eq:quintic3} such that the desingularization of $V (F) \subset \mathbf{P}^2$ is a superspecial trigonal curve of genus $5$. 
\end{prop}

\begin{proof}
Put $t = 11$, $u = 3$ and
\begin{eqnarray}
\{ p_1, \ldots , p_t \} & = & \{ x^4 z, x^3 y z, x^2 y^2 z, x y^3 z, y^4 z,  x^5, x^4 y, x^3 y^2, x^2 y^3, x y^4, y^5 \}, \nonumber \\
\{ q_1, \ldots , q_u \} & = & \{ (x^2 - \epsilon y^2 ) z^3,  x (x^2 + 3 \epsilon y^2 ) z^2, y (3 x^2 + \epsilon y^2 ) z^2 \}. \nonumber 
\end{eqnarray}
Execute Enumeration Algorithm 2 given in Subsection \ref{subsec:algorithm}.
We here give an outline of our computation together with our choices of $s_1$, $s_2$, $\{ k_1, \ldots , k_{s_1} \}$, $\{ i_1, \ldots , i_{s_2} \}$, $\mathcal{A}_1$, $\mathcal{A}_2$ and a term ordering in the algorithm.

First construct the quadratic extension field $K^{\prime}:=\mathbb{F}_{49} [T] / \langle T^2 - \epsilon \rangle \cong \mathbb{F}_{49^2}$, where we interpret $T = \sqrt{\epsilon}$.
Compute
\[
P = \begin{pmatrix}
1 & 1 & 0 & 0 & 0\\
\sqrt{\epsilon}^{-1}& -\sqrt{\epsilon}^{-1} & 0 & 0 & 0\\
0 & 0 & 1& 0 &0\\
0 & 0 &  0& 1 & 1 \\
0 & 0 & 0 & \sqrt{\epsilon}^{-1}& -\sqrt{\epsilon}^{-1}\\
\end{pmatrix}
\]
and $P^{-1}$.
Compute $P^{(7)}$, which is the matrix obtained by taking the $7$-th power of each entry of $P$.
\begin{enumerate}
\item[{\rm (0)}]
We set $s_1:= 11$, and $( k_1, \ldots , k_{s_1} ):= ( 1, \ldots , 11)$ (we regard the $11$ coefficients $a_1$, $a_2$, $a_3$, $a_4$, $a_5$, $a_6$, $a_7$, $a_8$, $a_9$, $a_{10}$ and $a_{11}$ as indeterminates).
Let $\mathcal{A}_1:=\emptyset$.
\end{enumerate}
We proceed with the following five steps:
\begin{enumerate}
	\item[{\rm (1)}] Compute $F:= \sum_{i=1}^t a_i p_i + \sum_{j=1}^u b_j q_j$ and $h:= ( F )^{p-1}$ over $\mathbb{F}_{49} [a_1, \ldots, a_{11}] [x, y, z]$, where $a_1, \ldots, a_{11}$ are indeterminates and $(b_1, b_2, b_3)=(1,1,0)$.
	\item[{\rm (2)}] We set $s_2 := 9$, and $(i_1, \ldots , i_{s_2}) := ( 1, 3, 5, 6, 7, 8, 9, 10, 11)$ (we regard the $9$ coefficients $a_1$, $a_3$, $a_5$, $a_6$, $a_7$, $a_8$, $a_9$, $a_{10}$ and $a_{11}$ as indeterminates).
	For the Gr\"{o}bner basis computation in $\mathbb{F}_{49} [ a_1, a_2, a_3, a_4, a_5, a_6, a_7, a_8, a_9, a_{10}, a_{11}]$ in (5), we adopt the graded reverse lexicographic (grevlex) order with
\[
a_{11} \prec a_{10} \prec a_9 \prec a_8 \prec a_7 \prec a_6 \prec a_5 \prec a_4 \prec a_3 \prec a_2 \prec a_1,
\]
	whereas for that in $\mathbb{F}_{49} [ x, y, z]$, the grevlex order with $z \prec y \prec x$ is adopted.
	Put $\mathcal{A}_2 := ( \mathbb{F}_{49} )^{\oplus 2}$.
	\item[{\rm (3)}] Substitute $\frac{X+Y}{2}$ and $\frac{X-Y}{-2\sqrt{\epsilon}}$ into $x$ and $y$ in $h$.
	Let $h^{\prime}$ denote the polynomial transformed from $h$ by the above substitution.
	\item[{\rm (4)}] Let $H^{\prime}$ be the Hasse-Witt matrix with respect to \eqref{BasisForH'}, i.e., the $(i,j)$-entry of $H^{\prime}$ ($1 \leq i,j \leq 5$) is the coefficient of the monomial $B_i B_j^{-p}$ in $h^{\prime}$.
	Here $B_i$ for $1 \leq i \leq 5$ are given in \eqref{BasisForH'}.
	\item[{\rm (5)}] Compute $H:=P^{(7)} H^{\prime} P^{-1}$.
	Let $\mathcal{S}$ be the set of the $25$ entries of $H$.
	Note that $\mathcal{S} \subset \mathbb{F}_{49} [a_1, \ldots, a_{11}]$.
	For each $( c_2, c_4 ) \in \mathcal{A}_2$, conduct procedures similar to the proof of Proposition \ref{prop:Case1q7}.
	\end{enumerate}
From the outputs, we have that there does not exist any quintic form $F$ of the form \eqref{eq:quintic3} such that the desingularization of $V (F) \subset \mathbf{P}^2$ is a superspecial trigonal curve of genus $5$. 
\end{proof}

\subsubsection{Non-split node case (2) and (3) with $q = p^2 = 49$}

\begin{prop}\label{prop:Case4q7}
Consider the quintic form
\begin{equation}
\begin{split}
F  = &  (x^2 - \epsilon y^2 ) z^3 + (c x^4 + a_2 x^3 y + a_3 x^2 y^2 + a_4 x y^3 + a_5 y^4) z \\
& + a_6 x^5 + a_7 x^4 y + a_8 x^3 y^2 + a_9 x^2 y^3 + a_{10} x y^4 + a_{11} y^5,
\end{split}\label{eq:quintic4}
\end{equation}
where $a_i \in \mathbb{F}_{49}$ for $2 \leq i \leq 11$, $c \in \{ 0, 1, \zeta^{(49)} \}$ and $\epsilon \in \mathbb{F}_{49}^{\times} \smallsetminus (\mathbb{F}_{49}^{\times})^2$.
Then there does not exist any quintic form $F$ of the form \eqref{eq:quintic4} such that the desingularization of $V (F) \subset \mathbf{P}^2$ is a superspecial trigonal curve of genus $5$. 
\end{prop}

\begin{proof}
Put $t = 11$, $u = 1$ and
\begin{eqnarray}
\{ p_1, \ldots , p_t \} & = & \{ x^4 z, x^3 y z, x^2 y^2 z, x y^3 z, y^4 z,  x^5, x^4 y, x^3 y^2, x^2 y^3, x y^4, y^5 \}, \nonumber \\
\{ q_1, \ldots , q_u \} & = & \{ (x^2 - \epsilon y^2 ) z^3 \}. \nonumber 
\end{eqnarray}
Execute Enumeration Algorithm 2 given in Subsection \ref{subsec:algorithm}.
We here give an outline of our computation together with our choices of $s_1$, $s_2$, $\{ k_1, \ldots , k_{s_1} \}$, $\{ i_1, \ldots , i_{s_2} \}$, $\mathcal{A}_1$, $\mathcal{A}_2$ and a term ordering in the algorithm.

First construct the quadratic extension field $K^{\prime}:=\mathbb{F}_{49} [T] / \langle T^2 - \epsilon \rangle \cong \mathbb{F}_{49^2}$, where we interpret $T = \sqrt{\epsilon}$.
Compute
\[
P = \begin{pmatrix}
1 & 1 & 0 & 0 & 0\\
\sqrt{\epsilon}^{-1}& -\sqrt{\epsilon}^{-1} & 0 & 0 & 0\\
0 & 0 & 1& 0 &0\\
0 & 0 &  0& 1 & 1 \\
0 & 0 & 0 & \sqrt{\epsilon}^{-1}& -\sqrt{\epsilon}^{-1}\\
\end{pmatrix}
\]
and $P^{-1}$.
Compute $P^{(7)}$, which is the matrix obtained by taking the $7$-th power of each entry of $P$.
\begin{enumerate}
\item[{\rm (0)}]
We set $s_1:= 11$, and $( k_1, \ldots , k_{s_1} ):= ( 1, \ldots , 11)$ (we regard the $11$ coefficients $a_1$, $a_2$, $a_3$, $a_4$, $a_5$, $a_6$, $a_7$, $a_8$, $a_9$, $a_{10}$ and $a_{11}$ as indeterminates).
Let $\mathcal{A}_1:=\emptyset$.
\end{enumerate}
We proceed with the following five steps:
\begin{enumerate}
	\item[{\rm (1)}] Compute $F:= \sum_{i=1}^t a_i p_i + \sum_{j=1}^u b_j q_j$ and $h:= ( F )^{p-1}$ over $\mathbb{F}_{49} [a_1, \ldots, a_{11}] [x, y, z]$, where $a_1, \ldots, a_{11}$ are indeterminates and $b_1=1$.
	\item[{\rm (2)}] We set $s_2 := 9$, and $(i_1, \ldots , i_{s_2}) := ( 3, 4, 5, 6, 7, 8, 9, 10, 11)$ (we regard the $9$ coefficients $a_3$, $a_4$, $a_5$, $a_6$, $a_7$, $a_8$, $a_9$, $a_{10}$ and $a_{11}$ as indeterminates).
	For the Gr\"{o}bner basis computation in $\mathbb{F}_{49} [ a_1, a_2, a_3, a_4, a_5, a_6, a_7, a_8, a_9, a_{10}, a_{11}]$ in (5), we adopt the graded reverse lexicographic (grevlex) order with
\[
a_{11} \prec a_{10} \prec a_9 \prec a_8 \prec a_7 \prec a_6 \prec a_5 \prec a_4 \prec a_3 \prec a_2 \prec a_1,
\]
	whereas for that in $\mathbb{F}_{49} [ x, y, z]$, the grevlex order with $z \prec y \prec x$ is adopted.
	Put $\mathcal{A}_2 := \{0, 1, \zeta^{(49)} \} \oplus \mathbb{F}_{49}$.
	\item[{\rm (3)}] Substitute $\frac{X+Y}{2}$ and $\frac{X-Y}{-2\sqrt{\epsilon}}$ into $x$ and $y$ in $h$.
	Let $h^{\prime}$ denote the polynomial transformed from $h$ by the above substitution.
	\item[{\rm (4)}] Let $H^{\prime}$ be the Hasse-Witt matrix with respect to \eqref{BasisForH'}, i.e., the $(i,j)$-entry of $H^{\prime}$ ($1 \leq i,j \leq 5$) is the coefficient of the monomial $B_i B_j^{-p}$ in $h^{\prime}$.
	\item[{\rm (5)}] Compute $H:=P^{(7)} H^{\prime} P^{-1}$.
	Let $\mathcal{S}$ be the set of the $25$ entries of $H$.
	Note that $\mathcal{S} \subset \mathbb{F}_{49} [a_1, \ldots, a_{11}]$.
	For each $( c_1, c_2 ) \in \mathcal{A}_2$, conduct procedures similar to the proof of Proposition \ref{prop:Case1q7}.
	\end{enumerate}
From the outputs, we have that there does not exist any quintic form $F$ of the form \eqref{eq:quintic4} such that the desingularization of $V (F) \subset \mathbf{P}^2$ is a superspecial trigonal curve of genus $5$. 
\end{proof}

\subsubsection{Cusp case with $q = p^2 = 49$}

\begin{prop}\label{prop:Case5q7}
Consider the quintic form
\begin{equation}
\begin{split}
F  = &  x^2 z^3 + a_1 y^3 z^2 + (a_2 x^4 + a_3 x^3 y + a_4 x^2 y^2 + b_1 x y^3 + a_5 y^4) z \\
& + a_6 x^5 + a_7 x^4 y + a_8 x^3 y^2 + a_9 x^2 y^3 + b_2 x y^4 + a_{10} y^5,
\end{split}\label{eq:quintic5}
\end{equation}
where $a_i \in \mathbb{F}_{49}$ for $1 \leq i \leq 10$ with $a_1 \neq 0$ and $(b_1, b_2) \in \{ (0,0), (1,0), (0,1), (1,1) \}$.
Then there does not exist any quintic form $F$ of the form \eqref{eq:quintic5} such that the desingularization of $V (F) \subset \mathbf{P}^2$ is a superspecial trigonal curve of genus $5$. 
\end{prop}

\begin{proof}
Put $t = 10$, $u = 3$ and
\begin{eqnarray}
\{ p_1, \ldots , p_t \} & = & \{ y^3 z^2, x^4 z, x^3 y z, x^2 y^2 z, y^4 z,  x^5, x^4 y, x^3 y^2, x^2 y^3, y^5 \}, \nonumber \\
\{ q_1, \ldots , q_u \} & = & \{ x y^3 z, x y^4,  x^2 z^3 \}. \nonumber 
\end{eqnarray}
For each $(b_1, b_2) \in \{ (0,0), (1,0), (0,1), (1,1) \}$, execute Enumeration Algorithm 1 given in Subsection \ref{subsec:algorithm}.
We here give an outline of our computation together with our choices of $s_1$, $s_2$, $\{ k_1, \ldots , k_{s_1} \}$, $\{ i_1, \ldots , i_{s_2} \}$, $\mathcal{A}_1$, $\mathcal{A}_2$ and a term ordering in the algorithm.
\begin{enumerate}
\item[{\rm (0)}]
We set $s_1:= 10$, and $( k_1, \ldots , k_{s_1} ):= ( 1, \ldots , 10)$ (we regard the $10$ coefficients $a_1$, $a_2$, $a_3$, $a_4$, $a_5$, $a_6$, $a_7$, $a_8$, $a_9$ and $a_{10}$ as indeterminates).
Let $\mathcal{A}_1:=\emptyset$.
\end{enumerate}
We proceed with the following three steps:
\begin{enumerate}
	\item[{\rm (1)}] Compute $F:= \sum_{i=1}^t a_i p_i + \sum_{j=1}^u b_j q_j$ and $h:= ( F )^{p-1}$ over $\mathbb{F}_{49} [a_1, \ldots, a_{10}] [x, y, z]$, where $a_1, \ldots, a_{10}$ are indeterminates and $b_3=1$. 
	\item[{\rm (2)}] We set $s_2 := 9$, and $(i_1, \ldots , i_{s_2}) := ( 2, 3, 4, 5, 6, 7, 8, 9, 10)$ (we regard the $9$ coefficients $a_2$, $a_3$, $a_4$, $a_5$, $a_6$, $a_7$, $a_8$, $a_9$ and $a_{10}$ as indeterminates).
	For the Gr\"{o}bner basis computation in $\mathbb{F}_{49} [ a_1, a_2, a_3, a_4, a_5, a_6, a_7, a_8, a_9, a_{10}]$ in (3), we adopt the graded reverse lexicographic (grevlex) order with
\[
a_{10} \prec a_9 \prec a_8 \prec a_7 \prec a_6 \prec a_5 \prec a_4 \prec a_3 \prec a_2 \prec a_1,
\]
	whereas for that in $\mathbb{F}_{49} [ x, y, z]$, the grevlex order with $z \prec y \prec x$ is adopted.
	Put $\mathcal{A}_2 := \mathbb{F}_{49}^{\times}$.
	\item[{\rm (3)}] Let $\mathcal{S} \subset \mathbb{F}_{49} [a_1, \ldots, a_{10}]$ be the set of the coefficients of the $25$ monomials in $h$, given in Corollary \ref{Criterion_ssp_non-split}.
	For each $c_1 \in \mathcal{A}_2$, conduct procedures similar to the proof of Proposition \ref{prop:Case1q7}.
\end{enumerate}
From the outputs, we have that there does not exist any quintic form $F$ of the form \eqref{eq:quintic5} such that the desingularization of $V (F) \subset \mathbf{P}^2$ is a superspecial trigonal curve of genus $5$.
\end{proof}

\subsubsection{Split node case (1) with $q = p = 11$}

\begin{prop}\label{prop:Case1q11}
Consider the quintic form
\begin{equation}
\begin{split}
F  = &  x y z^3 + (x^3 + b_1 y^3) z^2 + (a_1 x^4 + a_2 x^3 y + a_3 x^2 y^2 + a_4 x y^3 + a_5 y^4) z \\
& + a_6 x^5 + a_7 x^4 y + a_8 x^3 y^2 + a_9 x^2 y^3 + a_{10} x y^4 + a_{11} y^5,
\end{split}\label{eq:quintic1q11}
\end{equation}
where $a_i \in \mathbb{F}_{11}$ for $1 \leq i \leq 11$ and $b_1 \in \{ 0, 1 \}$.
Then there does not exist any quintic form $F$ of the form \eqref{eq:quintic1q11} such that the desingularization of $V (F) \subset \mathbf{P}^2$ is a superspecial trigonal curve of genus $5$.
\end{prop}

\begin{proof}
Put $t = 11$, $u = 3$ and
\begin{eqnarray}
\{ p_1, \ldots , p_t \} & = & \{ x^4 z, x^3 y z, x^2 y^2 z, x y^3 z, y^4 z,  x^5, x^4 y, x^3 y^2, x^2 y^3, x y^4, y^5 \}, \nonumber \\
\{ q_1, \ldots , q_u \} & = & \{ y^3 z^2, x^3 z^2, x y z^3 \}. \nonumber 
\end{eqnarray}
We divide our computation into the following two cases (this is our technical strategy to avoid the out of memory errors). 
\begin{description}
\item[{\bf (i) Case of $b_1 \neq 0$} (i.e., $b_1 = 1$).]
Execute Enumeration Algorithm 1 given in Subsection \ref{subsec:algorithm}.
We here give an outline of our computation together with our choices of $s_1$, $s_2$, $\{ k_1, \ldots , k_{s_1} \}$, $\{ i_1, \ldots , i_{s_2} \}$, $\mathcal{A}_1$, $\mathcal{A}_2$ and a term ordering in the algorithm.
\begin{enumerate}
\item[{\rm (0)}]
We set $s_1:= 11$, and $( k_1, \ldots , k_{s_1} ):= ( 1, \ldots , 11)$ (we regard the $11$ coefficients $a_1$, $a_2$, $a_3$, $a_4$, $a_5$, $a_6$, $a_7$, $a_8$, $a_9$, $a_{10}$ and $a_{11}$ as indeterminates).
Let $\mathcal{A}_1:=\emptyset$.
\end{enumerate}
We proceed with the following three steps:
\begin{enumerate}
	\item[{\rm (1)}] Compute $F:= \sum_{i=1}^t a_i p_i + \sum_{j=1}^u b_j q_j$ and $h:= ( F )^{p-1}$ over $\mathbb{F}_{11} [a_1, \ldots, a_{11}] [x, y, z]$, where $a_1, \ldots, a_{11}$ are indeterminates and $(b_2,b_3)=(1,1)$. 
	\item[{\rm (2)}] We set $s_2 := 8$, and $(i_1, \ldots , i_{s_2}) := ( 1, 2, 3, 6, 7, 8, 9, 10)$ (we regard the $8$ coefficients $a_1$, $a_2$, $a_3$, $a_6$, $a_7$, $a_8$, $a_9$ and $a_{10}$ as indeterminates).
	For the Gr\"{o}bner basis computation in $\mathbb{F}_{11} [ a_1, a_2, a_3, a_4, a_5, a_6, a_7, a_8, a_9, a_{10}, a_{11}]$ in (3), we adopt the graded reverse lexicographic (grevlex) order with
\[
a_{11} \prec a_{10} \prec a_9 \prec a_8 \prec a_7 \prec a_6 \prec a_5 \prec a_4 \prec a_3 \prec a_2 \prec a_1,
\]
	whereas for that in $\mathbb{F}_{11} [ x, y, z]$, the grevlex order with $z \prec y \prec x$ is adopted.
	Put $\mathcal{A}_2 := ( \mathbb{F}_{11} )^{\oplus 3}$.
	\item[{\rm (3)}] Let $\mathcal{S} \subset \mathbb{F}_{11} [a_1, \ldots, a_{11}]$ be the set of the coefficients of the $25$ monomials in $h$, given in Corollary \ref{Criterion_ssp_non-split}.
	For each $( c_4, c_5, c_{11} ) \in \mathcal{A}_2$, conduct procedures similar to the proof of Proposition \ref{prop:Case1q7}.
\end{enumerate}

\item[{\bf (ii) Case of $b_1 = 0$}.]
Execute Enumeration Algorithm 1 given in Subsection \ref{subsec:algorithm}.
We here give an outline of our computation together with our choices of $s_1$, $s_2$, $\{ k_1, \ldots , k_{s_1} \}$, $\{ i_1, \ldots , i_{s_2} \}$, $\mathcal{A}_1$, $\mathcal{A}_2$ and a term ordering in the algorithm.
\begin{enumerate}
\item[{\rm (0)}]
We set $s_1:= 11$, and $( k_1, \ldots , k_{s_1} ):= ( 1, \ldots , 11)$ (we regard the $11$ coefficients $a_1$, $a_2$, $a_3$, $a_4$, $a_5$, $a_6$, $a_7$, $a_8$, $a_9$, $a_{10}$ and $a_{11}$ as indeterminates).
Let $\mathcal{A}_1:=\emptyset$.
\end{enumerate}
We proceed with the following three steps:
\begin{enumerate}
	\item[{\rm (1)}] Compute $F:= \sum_{i=1}^t a_i p_i + \sum_{j=1}^u b_j q_j$ and $h:= ( F )^{p-1}$ over $\mathbb{F}_{11} [a_1, \ldots, a_{11}] [x, y, z]$, where $a_1, \ldots, a_{11}$ are indeterminates and $(b_2,b_3)=(1,1)$. 
	\item[{\rm (2)}] We set $s_2 := 7$, and $(i_1, \ldots , i_{s_2}) := ( 5, 6, 7, 8, 9, 10, 11)$ (we regard the $7$ coefficients $a_5$, $a_6$, $a_7$, $a_8$, $a_9$, $a_{10}$ and $a_{11}$ as indeterminates).
	For the Gr\"{o}bner basis computation in $\mathbb{F}_{11} [ a_1, a_2, a_3, a_4, a_5, a_6, a_7, a_8, a_9, a_{10}, a_{11}]$ in (3), we adopt the graded reverse lexicographic (grevlex) order with
\[
a_{11} \prec a_{10} \prec a_9 \prec a_8 \prec a_7 \prec a_6 \prec a_5 \prec a_4 \prec a_3 \prec a_2 \prec a_1,
\]
	whereas for that in $\mathbb{F}_{11} [ x, y, z]$, the grevlex order with $z \prec y \prec x$ is adopted.
	Put $\mathcal{A}_2 := ( \mathbb{F}_{11} )^{\oplus 4}$.
	\item[{\rm (3)}] Let $\mathcal{S} \subset \mathbb{F}_{11} [a_1, \ldots, a_{11}]$ be the set of the coefficients of the $25$ monomials in $h$, given in Corollary \ref{Criterion_ssp_non-split}.
	For each $( c_1, c_2, c_3, c_4 ) \in \mathcal{A}_2$, conduct procedures similar to the proof of Proposition \ref{prop:Case1q7}.
\end{enumerate}
\end{description}
From the outputs, we have that there does not exist any quintic form $F$ of the form \eqref{eq:quintic1q11} such that the desingularization of $V (F) \subset \mathbf{P}^2$ is a superspecial trigonal curve of genus $5$. 
\end{proof}

\subsubsection{Split node case (2) with $q = p = 11$}

\begin{prop}\label{prop:Case2q11}
Consider the quintic form
\begin{equation}
\begin{split}
F  = &  x y z^3 + (c_1 x^4 + c_2 x^3 y + a_3 x^2 y^2 + a_4 x y^3 + a_5 y^4) z \\
& + a_6 x^5 + a_7 x^4 y + a_8 x^3 y^2 + a_9 x^2 y^3 + a_{10} x y^4 + a_{11} y^5,
\end{split}\label{eq:quintic2q11}
\end{equation}
where $a_i \in \mathbb{F}_{11}$ for $3 \leq i \leq 11$ and $(c_1, c_2) \in \{ (0,0), (1,0), (0,1), (1,1), (1,\zeta^{(11)}) \}$.
Then the desingularization of $V (F) \subset \mathbf{P}^2$ is superspecial if and only if $a_6, a_{11} \in \F_{11}^\times$, $a_i = 0$ for $i =3, \ldots , 5, 7, \ldots 10$ and $c_1 = c_2 = 0$.
\end{prop}

\begin{proof}
Put $t = 11$, $u = 1$ and
\begin{eqnarray}
\{ p_1, \ldots , p_t \} & = & \{  x^4 z, x^3 y z, x^2 y^2 z, x y^3 z, y^4 z,  x^5, x^4 y, x^3 y^2, x^2 y^3, x y^4, y^5 \}, \nonumber \\
\{ q_1, \ldots , q_u \} & = & \{ x y z^3 \}. \nonumber 
\end{eqnarray}
Execute Enumeration Algorithm 1 given in Subsection \ref{subsec:algorithm}.
We here give an outline of our computation together with our choices of $s_1$, $s_2$, $\{ k_1, \ldots , k_{s_1} \}$, $\{ i_1, \ldots , i_{s_2} \}$, $\mathcal{A}_1$, $\mathcal{A}_2$ and a term ordering in the algorithm.
\begin{enumerate}
\item[{\rm (0)}]
We set $s_1:= 11$, and $( k_1, \ldots , k_{s_1} ):= ( 1, \ldots , 11)$ (we regard the $11$ coefficients $a_1$, $a_2$, $a_3$, $a_4$, $a_5$, $a_6$, $a_7$, $a_8$, $a_9$, $a_{10}$ and $a_{11}$ as indeterminates).
Let $\mathcal{A}_1:=\emptyset$.
\end{enumerate}
We proceed with the following three steps:
\begin{enumerate}
	\item[{\rm (1)}] Compute $F:= \sum_{i=1}^t a_i p_i + \sum_{j=1}^u b_j q_j$ and $h:= ( F )^{p-1}$ over $\mathbb{F}_{11} [a_1, \ldots, a_{11}] [x, y, z]$, where $a_1, \ldots, a_{11}$ are indeterminates and $b_1=1$. 
	\item[{\rm (2)}] We set $s_2 := 7$, and $(i_1, \ldots , i_{s_2}) := ( 5, 6, 7, 8, 9, 10, 11)$ (we regard the $7$ coefficients $a_5$, $a_6$, $a_7$, $a_8$, $a_9$, $a_{10}$ and $a_{11}$ as indeterminates).
	For the Gr\"{o}bner basis computation in the polynomial ring $\mathbb{F}_{11} [ a_1, a_2, a_3, a_4, a_5, a_6, a_7, a_8, a_9, a_{10}, a_{11}]$ in (3), we adopt the graded reverse lexicographic (grevlex) order with
\[
a_{11} \prec a_{10} \prec a_9 \prec a_8 \prec a_7 \prec a_6 \prec a_5 \prec a_4 \prec a_3 \prec a_2 \prec a_1.
\]
	Put $\mathcal{A}_2 := \{ (0,0), (1,0), (0,1), (1,1), (1,\zeta^{(11)}) \} \oplus (\mathbb{F}_{11})^{\oplus 2}$.
	\item[{\rm (3)}] Let $\mathcal{S} \subset \mathbb{F}_{11} [a_1, \ldots, a_{11}]$ be the set of the coefficients of the $25$ monomials in $h$, given in Corollary \ref{Criterion_ssp_non-split}.
	For each $( c_1, c_2, c_3, c_4 ) \in \mathcal{A}_2$, conduct procedures similar to the proof of Proposition \ref{prop:Case1q7}.
\end{enumerate}
Each quintic form obtained as an element of the output is irreducible over $\overline{\mathbb{F}_{11}}$, which we check by the method for the irreducibility test given in Appendix \ref{sec:irr}.
From the outputs, we have that the desingularization of $V (F) \subset \mathbf{P}^2$ is superspecial if and only if $a_6, a_{11} \in \F_{11}^\times$, $a_i = 0$ for $i =3, \ldots , 5, 7, \ldots 10$ and $c_1 = c_2 = 0$.
\end{proof}

\subsubsection{Non-split node case (1) with $q = p= 11$}

\begin{prop}\label{prop:Case3q11}
Consider the quintic form
\begin{equation}
\begin{split}
F  = &  (x^2 - \epsilon y^2 ) z^3 + x (x^2 + 3 \epsilon y^2 ) z^2 + b y (3 x^2 + \epsilon y^2) ) z^2 \\
& + (a_1 x^4 + a_2 x^3 y + a_3 x^2 y^2 + a_4 x y^3 + a_5 y^4) z \\
& + a_6 x^5 + a_7 x^4 y + a_8 x^3 y^2 + a_9 x^2 y^3 + a_{10} x y^4 + a_{11} y^5,
\end{split}\label{eq:quintic3q11}
\end{equation}
where $a_i \in \mathbb{F}_{11}$ for $1 \leq i \leq 11$, $b \in \{ 0, 6, 10 \}$ and $\epsilon \in \mathbb{F}_{11}^{\times} \smallsetminus (\mathbb{F}_{11}^{\times})^2$.
Then there does not exist any quintic form $F$ of the form \eqref{eq:quintic3q11} such that the desingularization of $V (F) \subset \mathbf{P}^2$ is superspecial. 
\end{prop}

\begin{proof}
Put $t = 11$, $u = 3$ and
\begin{eqnarray}
\{ p_1, \ldots , p_t \} & = & \{ x^4 z, x^3 y z, x^2 y^2 z, x y^3 z, y^4 z,  x^5, x^4 y, x^3 y^2, x^2 y^3, x y^4, y^5 \}, \nonumber \\
\{ q_1, \ldots , q_u \} & = & \{ y (3 x^2 + \epsilon y^2) ) z^2, (x^2 - \epsilon y^2 ) z^3,  x (x^2 + 3 \epsilon y^2 ) z^2 \}. \nonumber 
\end{eqnarray}
For each $b \in \{ 0, 6, 10 \}$, execute Enumeration Algorithm 2 given in Subsection \ref{subsec:algorithm}.
We here give an outline of our computation together with our choices of $s_1$, $s_2$, $\{ k_1, \ldots , k_{s_1} \}$, $\{ i_1, \ldots , i_{s_2} \}$, $\mathcal{A}_1$, $\mathcal{A}_2$ and a term ordering in the algorithm.

First construct the quadratic extension field $K^{\prime}:=\mathbb{F}_{11} [T] / \langle T^2 - \epsilon \rangle \cong \mathbb{F}_{121}$, where we interpret $T = \sqrt{\epsilon}$.
Compute
\[
P = \begin{pmatrix}
1 & 1 & 0 & 0 & 0\\
\sqrt{\epsilon}^{-1}& -\sqrt{\epsilon}^{-1} & 0 & 0 & 0\\
0 & 0 & 1& 0 &0\\
0 & 0 &  0& 1 & 1 \\
0 & 0 & 0 & \sqrt{\epsilon}^{-1}& -\sqrt{\epsilon}^{-1}\\
\end{pmatrix}
\]
and $P^{-1}$.
Compute $P^{(11)}$, which is the matrix obtained by taking the $11$-th power of each entry of $P$.
\begin{enumerate}
\item[{\rm (0)}]
We set $s_1:= 11$, and $( k_1, \ldots , k_{s_1} ):= ( 1, \ldots , 11)$ (we regard the $11$ coefficients $a_1$, $a_2$, $a_3$, $a_4$, $a_5$, $a_6$, $a_7$, $a_8$, $a_9$, $a_{10}$ and $a_{11}$ as indeterminates).
Let $\mathcal{A}_1:=\emptyset$.
\end{enumerate}
We proceed with the following five steps:
\begin{enumerate}
	\item[{\rm (1)}] Compute $F:= \sum_{i=1}^t a_i p_i + \sum_{j=1}^u b_j q_j$ and $h:= ( F )^{p-1}$ over $\mathbb{F}_{11} [a_1, \ldots, a_{11}] [x, y, z]$, where $a_1, \ldots, a_{11}$ are indeterminates and $(b_1,b_2,b_3)=(b,1,1)$.
	\item[{\rm (2)}] We set $s_2 := 6$, and $(i_1, \ldots , i_{s_2}) := ( 6, 7, 8, 9, 10, 11)$ (we regard the $6$ coefficients $a_6$, $a_7$, $a_8$, $a_9$, $a_{10}$ and $a_{11}$ as indeterminates).
	For the Gr\"{o}bner basis computation in the polynomial ring $\mathbb{F}_{11} [ a_1, a_2, a_3, a_4, a_5, a_6, a_7, a_8, a_9, a_{10}, a_{11}]$ in (5), we adopt the graded reverse lexicographic (grevlex) order with
\[
a_{11} \prec a_{10} \prec a_9 \prec a_8 \prec a_7 \prec a_6 \prec a_{5} \prec a_4 \prec a_3 \prec a_2 \prec a_1,
\]
	whereas for that in $\mathbb{F}_{11} [ x, y, z]$, the grevlex order with $z \prec y \prec x$ is adopted.
	Put $\mathcal{A}_2 := ( \mathbb{F}_{11} )^{\oplus 5}$.
	\item[{\rm (3)}] Substitute $\frac{X+Y}{2}$ and $\frac{X-Y}{-2\sqrt{\epsilon}}$ into $x$ and $y$ in $h$.
	Let $h^{\prime}$ denote the polynomial transformed from $h$ by the above substitution.
	\item[{\rm (4)}] Let $H^{\prime}$ be the Hasse-Witt matrix with respect to \eqref{BasisForH'}, i.e., the $(i,j)$-entry of $H^{\prime}$ ($1 \leq i,j \leq 5$) is the coefficient of the monomial $B_i B_j^{-p}$ in $h^{\prime}$.
	\item[{\rm (5)}] Compute $H:=P^{(11)} H^{\prime} P^{-1}$.
	Let $\mathcal{S}$ be the set of the $25$ entries of $H$.
	Note that $\mathcal{S} \subset \mathbb{F}_{11} [a_1, \ldots, a_{11}]$.
	For each $( c_1, c_2, c_3, c_4, c_5 ) \in \mathcal{A}_2$, conduct procedures similar to the proof of Proposition \ref{prop:Case1q7}.
	\end{enumerate}
From the outputs, we have that there does not exist any quintic form $F \in \mathbb{F}_{11}[x,y,z]$ of the form \eqref{eq:quintic3q11} such that the desingularization of $V (F) \subset \mathbf{P}^2$ is a superspecial trigonal curve of genus $5$.
\end{proof}

\subsubsection{Non-split node case (2) and (3) with $q = p= 11$}

\begin{prop}\label{prop:Case4q11}
Consider the quintic form
\begin{equation}
\begin{split}
F  = &  (x^2 - \epsilon y^2 ) z^3  + (c x^4 + a_2 x^3 y + a_3 x^2 y^2 + a_4 x y^3 + a_5 y^4) z \\
& + a_6 x^5 + a_7 x^4 y + a_8 x^3 y^2 + a_9 x^2 y^3 + a_{10} x y^4 + a_{11} y^5,
\end{split}\label{eq:quintic4q11}
\end{equation}
where $a_i \in \mathbb{F}_{11}$ for $2 \leq i \leq 11$, $c \in \{ 0, 1, \zeta^{(11)} \}$ and $\epsilon \in \mathbb{F}_{11}^{\times} \smallsetminus (\mathbb{F}_{11}^{\times})^2$.
Then the desingularization of $V (F)$ is superspecial if and only if $(a_6, a_{7}) \in (\F_{11})^{\oplus 2} \smallsetminus \{ (0,0) \}$, $a_8 = a_{10} = 9 a_6$, $a_9 = 4 a_7$, $a_{11}= 3 a_7$, $a_i = 0$ for $i =2, 3, 4 , 5$ and $c = 0$.
\end{prop}

\begin{proof}
Put $t = 11$, $u = 1$ and
\begin{eqnarray}
\{ p_1, \ldots , p_t \} & = & \{ x^4 z, x^3 y z, x^2 y^2 z, x y^3 z, y^4 z,  x^5, x^4 y, x^3 y^2, x^2 y^3, x y^4, y^5 \}, \nonumber \\
\{ q_1, \ldots , q_u \} & = & \{ (x^2 - \epsilon y^2 ) z^3 \}. \nonumber 
\end{eqnarray}
Execute Enumeration Algorithm 2 given in Subsection \ref{subsec:algorithm}.
We here give an outline of our computation together with our choices of $s_1$, $s_2$, $\{ k_1, \ldots , k_{s_1} \}$, $\{ i_1, \ldots , i_{s_2} \}$, $\mathcal{A}_1$, $\mathcal{A}_2$ and a term ordering in the algorithm.

First construct the quadratic extension field $K^{\prime}:=\mathbb{F}_{11} [T] / \langle T^2 - \epsilon \rangle \cong \mathbb{F}_{121}$, where we interpret $T = \sqrt{\epsilon}$.
Compute
\[
P = \begin{pmatrix}
1 & 1 & 0 & 0 & 0\\
\sqrt{\epsilon}^{-1}& -\sqrt{\epsilon}^{-1} & 0 & 0 & 0\\
0 & 0 & 1& 0 &0\\
0 & 0 &  0& 1 & 1 \\
0 & 0 & 0 & \sqrt{\epsilon}^{-1}& -\sqrt{\epsilon}^{-1}\\
\end{pmatrix}
\]
and $P^{-1}$.
Compute $P^{(11)}$, which is the matrix obtained by taking the $11$-th power of each entry of $P$.
\begin{enumerate}
\item[{\rm (0)}]
We set $s_1:= 11$, and $( k_1, \ldots , k_{s_1} ):= ( 1, \ldots , 11)$ (we regard the $11$ coefficients $a_1$, $a_2$, $a_3$, $a_4$, $a_5$, $a_6$, $a_7$, $a_8$, $a_9$, $a_{10}$ and $a_{11}$ as indeterminates).
Let $\mathcal{A}_1:=\emptyset$.
\end{enumerate}
We proceed with the following five steps:
\begin{enumerate}
	\item[{\rm (1)}] Compute $F:= \sum_{i=1}^t a_i p_i + \sum_{j=1}^u b_j q_j$ and $h:= ( F )^{p-1}$ over $\mathbb{F}_{11} [a_1, \ldots, a_{11}] [x, y, z]$, where $a_1, \ldots, a_{11}$ are indeterminates and $b_1=1$.
	\item[{\rm (2)}] We set $s_2 :=6 $, and $(i_1, \ldots , i_{s_2}) := ( 6, 7, 8, 9, 10, 11)$ (we regard the $6$ coefficients $a_6$, $a_7$, $a_8$, $a_9$, $a_{10}$ and $a_{11}$ as indeterminates).
	For the Gr\"{o}bner basis computation in the polynomial ring $\mathbb{F}_{11} [ a_1, a_2, a_3, a_4, a_5, a_6, a_7, a_8, a_9, a_{10}, a_{11}]$ in (5), we adopt the graded reverse lexicographic (grevlex) order with
\[
a_{11} \prec a_{10} \prec a_9 \prec a_8 \prec a_7 \prec a_6 \prec a_{5} \prec a_4 \prec a_3 \prec a_2 \prec a_1,
\]
	whereas for that in $\mathbb{F}_{11} [ x, y, z]$, the grevlex order with $z \prec y \prec x$ is adopted.
	Put $\mathcal{A}_2 := \{0, 1, \zeta^{(11)} \} \oplus (\mathbb{F}_{11})^{\oplus 4}$.
	\item[{\rm (3)}] Substitute $\frac{X+Y}{2}$ and $\frac{X-Y}{-2\sqrt{\epsilon}}$ into $x$ and $y$ in $h$.
	Let $h^{\prime}$ denote the polynomial transformed from $h$ by the above substitution.
	\item[{\rm (4)}] Let $H^{\prime}$ be the Hasse-Witt matrix with respect to \eqref{BasisForH'}, i.e., the $(i,j)$-entry of $H^{\prime}$ ($1 \leq i,j \leq 5$) is the coefficient of the monomial $B_i B_j^{-p}$ in $h^{\prime}$.
	\item[{\rm (5)}] Compute $H:=P^{(11)} H^{\prime} P^{-1}$.
	Let $\mathcal{S}$ be the set of the $25$ entries of $H$.
	Note that $\mathcal{S} \subset \mathbb{F}_{11} [a_1, \ldots, a_{11}]$.
	For each $( c_1, c_2, c_3, c_4, c_5 ) \in \mathcal{A}_2$, conduct procedures similar to the proof of Proposition \ref{prop:Case1q7}.
	\end{enumerate}
	Each quintic form obtained as an element of the output is irreducible over $\overline{\mathbb{F}_{11}}$, which we check by the method for the irreducibility test given in Appendix \ref{sec:irr}.
From the outputs, we have that the desingularization of $V (F)$ is superspecial if and only if $(a_6, a_{7}) \in (\F_{11})^{\oplus 2} \smallsetminus \{ (0,0) \}$, $a_8 = a_{10} = 9 a_6$, $a_9 = 4 a_7$, $a_{11}= 3 a_7$, $a_i = 0$ for $i =2, 3, 4 , 5$ and $c = 0$.
\end{proof}

\subsubsection{Cusp case with $q = p=11$}

\begin{prop}\label{prop:Case5q11}
Consider the quintic form
\begin{equation}
\begin{split}
F  = &  x^2 z^3 + a_1 y^3 z^2 + (a_2 x^4 + a_3 x^3 y + a_4 x^2 y^2 + b_1 x y^3 + a_5 y^4) z \\
& + a_6 x^5 + a_7 x^4 y + a_8 x^3 y^2 + a_9 x^2 y^3 + b_2 x y^4 + a_{10} y^5,
\end{split}\label{eq:quintic5q11}
\end{equation}
where $a_i \in \mathbb{F}_{11}$ for $1 \leq i \leq 10$ with $a_1 \neq 0$ and $(b_1, b_2) \in \{ (0,0), (1,0), (0,1), (1,1) \}$.
Then there does not exist any quintic form $F$ of the form \eqref{eq:quintic5q11} such that the desingularization of $V (F) \subset \mathbf{P}^2$ is a superspecial trigonal curve of genus $5$. 
\end{prop}

\begin{proof}
Put $t = 10$, $u = 3$ and
\begin{eqnarray}
\{ p_1, \ldots , p_t \} & = & \{ y^3 z^2, x^4 z, x^3 y z, x^2 y^2 z, y^4 z,  x^5, x^4 y, x^3 y^2, x^2 y^3, y^5 \}, \nonumber \\
\{ q_1, \ldots , q_u \} & = & \{ x y^3 z, x y^4, x^2 z^3 \}. \nonumber 
\end{eqnarray}
For each $(b_1, b_2) \in \{ (0,0), (1,0), (0,1), (1,1) \}$, execute Enumeration Algorithm 1 given in Subsection \ref{subsec:algorithm}.
We here give an outline of our computation together with our choices of $s_1$, $s_2$, $\{ k_1, \ldots , k_{s_1} \}$, $\{ i_1, \ldots , i_{s_2} \}$, $\mathcal{A}_1$, $\mathcal{A}_2$ and a term ordering in the algorithm.
\begin{enumerate}
\item[{\rm (0)}]
We set $s_1:= 10$, and $( k_1, \ldots , k_{s_1} ):= ( 1, \ldots , 10)$ (we regard the $10$ coefficients $a_1$, $a_2$, $a_3$, $a_4$, $a_5$, $a_6$, $a_7$, $a_8$, $a_9$ and $a_{10}$ as indeterminates).
Let $\mathcal{A}_1:=\emptyset$.
\end{enumerate}
We proceed with the following three steps:
\begin{enumerate}
	\item[{\rm (1)}] Compute $F:= \sum_{i=1}^t a_i p_i + \sum_{j=1}^u b_j q_j$ and $h:= ( F )^{p-1}$ over $\mathbb{F}_{11} [a_1, \ldots, a_{10}] [x, y, z]$, where $a_1, \ldots, a_{10}$ are indeterminates and $b_3=1$. 
	\item[{\rm (2)}] We set $s_2 := 7$, and $(i_1, \ldots , i_{s_2}) := ( 2, 4, 6, 7, 8, 9, 10)$ (we regard the $7$ coefficients $a_2$, $a_4$, $a_6$, $a_7$, $a_8$, $a_9$ and $a_{10}$ as indeterminates).
	For the Gr\"{o}bner basis computation in the polynomial ring $\mathbb{F}_{11} [ a_1, a_2, a_3, a_4, a_5, a_6, a_7, a_8, a_9, a_{10}]$ in (3), we adopt the graded reverse lexicographic (grevlex) order with
\[
a_{10} \prec a_9 \prec a_8 \prec a_7 \prec a_6 \prec a_5 \prec a_4 \prec a_3 \prec a_2 \prec a_1,
\]
	whereas for that in $\mathbb{F}_{11} [ x, y, z]$, the grevlex order with $z \prec y \prec x$ is adopted.
	Put $\mathcal{A}_2 := \mathbb{F}_{11}^{\times} \oplus (\mathbb{F}_{11})^{\oplus 2}$.
	\item[{\rm (3)}] Let $\mathcal{S} \subset \mathbb{F}_{11} [a_1, \ldots, a_{10}]$ be the set of the coefficients of the $25$ monomials in $h$, given in Corollary \ref{Criterion_ssp_non-split}.
	For each $( c_1, c_3, c_5 ) \in \mathcal{A}_2$, conduct procedures similar to the proof of Proposition \ref{prop:Case1q7}.
\end{enumerate}
From the outputs, we have that there does not exist any quintic form $F \in \mathbb{F}_{11}[x,y,z]$ of the form \eqref{eq:quintic5q11} such that the desingularization of $V (F) \subset \mathbf{P}^2$ is a superspecial trigonal curve of genus $5$.
\end{proof}

\subsubsection{Split node case (1) with $q = p=13$}

\begin{prop}\label{prop:Case1q13}
Consider the quintic form
\begin{equation}
\begin{split}
F  = &  x y z^3 + (x^3 + b_1 y^3) z^2 + (a_1 x^4 + a_2 x^3 y + a_3 x^2 y^2 + a_4 x y^3 + a_5 y^4) z \\
& + a_6 x^5 + a_7 x^4 y + a_8 x^3 y^2 + a_9 x^2 y^3 + a_{10} x y^4 + a_{11} y^5,
\end{split}\label{eq:quintic1q13}
\end{equation}
where $a_i \in \mathbb{F}_{13}$ for $1 \leq i \leq 11$ and $b_1 \in \{ 0, 1 , \zeta^{(13)} \}$.
Then there does not exist any quintic form $F$ of the form \eqref{eq:quintic1q13} such that the desingularization of $V (F) \subset \mathbf{P}^2$ is a superspecial trigonal curve of genus $5$. 
\end{prop}

\begin{proof}
Put $t = 11$, $u = 3$ and
\begin{eqnarray}
\{ p_1, \ldots , p_t \} & = & \{ x^4 z, x^3 y z, x^2 y^2 z, x y^3 z, y^4 z,  x^5, x^4 y, x^3 y^2, x^2 y^3, x y^4, y^5 \}, \nonumber \\
\{ q_1, \ldots , q_u \} & = & \{ y^3 z^2, x y z^3, x^3 z^2 \}. \nonumber 
\end{eqnarray}
For each $b_1 \in \{ 0, 1, \zeta^{(13)} \}$, execute Enumeration Algorithm 1 given in Subsection \ref{subsec:algorithm}.
We here give an outline of our computation together with our choices of $s_1$, $s_2$, $\{ k_1, \ldots , k_{s_1} \}$, $\{ i_1, \ldots , i_{s_2} \}$, $\mathcal{A}_1$, $\mathcal{A}_2$ and a term ordering in the algorithm.
\begin{enumerate}
\item[{\rm (0)}]
We set $s_1:= 10$, and $( k_1, \ldots , k_{s_1} ):= ( 2, \ldots , 11)$ (we regard the $10$ coefficients $a_2$, $a_3$, $a_4$, $a_5$, $a_6$, $a_7$, $a_8$, $a_9$, $a_{10}$ and $a_{11}$ as indeterminates).
Let $\mathcal{A}_1:=\mathbb{F}_{13}$.
\end{enumerate}
For each $c_1 \in \mathcal{A}_1$, we proceed with the following three steps:
\begin{enumerate}
	\item[{\rm (1)}] Substitute $c_1$ into $a_1$ in $F$, and compute $F:= \sum_{i=1}^t a_i p_i + \sum_{j=1}^u b_j q_j$ and $h:= ( F )^{p-1}$ over $\mathbb{F}_{13} [a_2, \ldots, a_{11}] [x, y, z]$, where $a_2, \ldots, a_{11}$ are indeterminates and $(b_2,b_3)=(1,1)$. 
	\item[{\rm (2)}] We set $s_2 := 6$, and $(i_1, \ldots , i_{s_2}) := (  6, 7, 8, 9, 10, 11)$ (we regard the $6$ coefficients $a_6$, $a_7$, $a_8$, $a_9$, $a_{10}$ and $a_{11}$ as indeterminates).
	For the Gr\"{o}bner basis computation in the polynomial ring $\mathbb{F}_{13} [ a_2, a_3, a_4, a_5, a_6, a_7, a_8, a_9, a_{10}, a_{11}]$ in (3), we adopt the graded reverse lexicographic (grevlex) order with
\[
a_{11} \prec a_{10} \prec a_9 \prec a_8 \prec a_7 \prec a_6 \prec a_{5} \prec a_4 \prec a_3 \prec a_2,
\]
	whereas for that in $\mathbb{F}_{13} [ x, y, z]$, the grevlex order with $z \prec y \prec x$ is adopted.
	Put $\mathcal{A}_2 := ( \mathbb{F}_{13} )^{\oplus 4}$.
	\item[{\rm (3)}] Let $\mathcal{S} \subset \mathbb{F}_{13} [a_2, \ldots, a_{11}]$ be the set of the coefficients of the $25$ monomials in $h$, given in Corollary \ref{Criterion_ssp_non-split}.
	For each $( c_2, c_3, c_4, c_5 ) \in \mathcal{A}_2$, conduct procedures similar to the proof of Proposition \ref{prop:Case1q7}.
\end{enumerate}
From the outputs, we have that there does not exist any quintic form $F$ of the form \eqref{eq:quintic1q13} such that the desingularization of $V (F) \subset \mathbf{P}^2$ is a superspecial trigonal curve of genus $5$. 
\end{proof}

\begin{rem}
When we executed the computation described in the proof of Proposition \ref{prop:Case1q13} over Magma, we divided our computation program into two files (this is our technical strategy to avoid the out of memory errors).
One of them is a code for the case $b_1 \neq 0$ (i.e., $b_1 \in \{ 1, \zeta^{(13)} \}$), and the other is a code for the case $b_1=0$.
\end{rem}

\subsubsection{Split node case (2) with $q = p=13$}

\begin{prop}\label{prop:Case2q13}
Consider the quintic form
\begin{equation}
\begin{split}
F  = &  x y z^3 + (c_1 x^4 + c_2 x^3 y + a_3 x^2 y^2 + a_4 x y^3 + a_5 y^4) z \\
& + a_6 x^5 + a_7 x^4 y + a_8 x^3 y^2 + a_9 x^2 y^3 + a_{10} x y^4 + a_{11} y^5,
\end{split}\label{eq:quintic2q13}
\end{equation}
where $a_i \in \mathbb{F}_{13}$ for $3 \leq i \leq 11$ and $(c_1, c_2) \in \{ (0,0), (1,0), (0,1), (1,1), (1,\zeta^{(13)}) \}$.
Then there does not exist any quintic form $F$ of the form \eqref{eq:quintic2q13} such that the desingularization of $V (F) \subset \mathbf{P}^2$ is a superspecial trigonal curve of genus $5$. 
\end{prop}

\begin{proof}
Put $t = 11$, $u = 1$ and
\begin{eqnarray}
\{ p_1, \ldots , p_t \} & = & \{  x^4 z, x^3 y z, x^2 y^2 z, x y^3 z, y^4 z,  x^5, x^4 y, x^3 y^2, x^2 y^3, x y^4, y^5 \}, \nonumber \\
\{ q_1, \ldots , q_u \} & = & \{ x y z^3 \}. \nonumber 
\end{eqnarray}
Execute Enumeration Algorithm 1 given in Subsection \ref{subsec:algorithm}.
We here give an outline of our computation together with our choices of $s_1$, $s_2$, $\{ k_1, \ldots , k_{s_1} \}$, $\{ i_1, \ldots , i_{s_2} \}$, $\mathcal{A}_1$, $\mathcal{A}_2$ and a term ordering in the algorithm.
\begin{enumerate}
\item[{\rm (0)}]
We set $s_1:= 11$, and $( k_1, \ldots , k_{s_1} ):= ( 1, \ldots , 11)$ (we regard the $11$ coefficients $a_1$, $a_2$, $a_3$, $a_4$, $a_5$, $a_6$, $a_7$, $a_8$, $a_9$, $a_{10}$ and $a_{11}$ as indeterminates).
Let $\mathcal{A}_1:=\emptyset$.
\end{enumerate}
We proceed with the following three steps:
\begin{enumerate}
	\item[{\rm (1)}] Compute $F:= \sum_{i=1}^t a_i p_i + \sum_{j=1}^u b_j q_j$ and $h:= ( F )^{p-1}$ over $\mathbb{F}_{13} [a_1, \ldots, a_{11}] [x, y, z]$, where $a_1, \ldots, a_{11}$ are indeterminates and $b_1=1$. 
	\item[{\rm (2)}] We set $s_2 := 6$, and $(i_1, \ldots , i_{s_2}) := ( 6, 7, 8, 9, 10, 11)$ (we regard the $6$ coefficients $a_6$, $a_7$, $a_8$, $a_9$, $a_{10}$ and $a_{11}$ as indeterminates).
	For the Gr\"{o}bner basis computation in the polynomial ring $\mathbb{F}_{13} [ a_1, a_2, a_3, a_4, a_5, a_6, a_7, a_8, a_9, a_{10}, a_{11}]$ in (3), we adopt the graded reverse lexicographic (grevlex) order with
\[
a_{11} \prec a_{10} \prec a_9 \prec a_8 \prec a_7 \prec a_6 \prec a_{5} \prec a_4 \prec a_3 \prec a_2 \prec a_1,
\]
	whereas for that in $\mathbb{F}_{13} [ x, y, z]$, the grevlex order with $z \prec y \prec x$ is adopted.
	Put $\mathcal{A}_2 := \{ (0,0), (1,0), (0,1), (1,1), (1,\zeta^{(13)}) \} \oplus (\mathbb{F}_{13})^{\oplus 3}$.
	\item[{\rm (3)}] Let $\mathcal{S} \subset \mathbb{F}_{13} [a_1, \ldots, a_{11}]$ be the set of the coefficients of the $25$ monomials in $h$, given in Corollary \ref{Criterion_ssp_non-split}.
	For each $( c_1, c_2, c_3, c_4, c_5 ) \in \mathcal{A}_2$, conduct procedures similar to the proof of Proposition \ref{prop:Case1q7}.
\end{enumerate}
From the outputs, we have that there does not exist any quintic form $F$ of the form \eqref{eq:quintic2q13} such that the desingularization of $V (F) \subset \mathbf{P}^2$ is a superspecial trigonal curve of genus $5$. 
\end{proof}

\subsubsection{Non-split node case (1) with $q = p=13$}

\begin{prop}\label{prop:Case3q13}
Consider the quintic form
\begin{equation}
\begin{split}
F  = &  (x^2 - \epsilon y^2 ) z^3 + x (x^2 + 3 \epsilon y^2 ) z^2 + 0 \cdot y (3 x^2 + \epsilon y^2 ) z^2 \\
& + (a_1 x^4 + a_2 x^3 y + a_3 x^2 y^2 + a_4 x y^3 + a_5 y^4) z \\
& + a_6 x^5 + a_7 x^4 y + a_8 x^3 y^2 + a_9 x^2 y^3 + a_{10} x y^4 + a_{11} y^5,
\end{split}\label{eq:quintic3q13}
\end{equation}
where $a_i \in \mathbb{F}_{13}$ for $1 \leq i \leq 11$ and $\epsilon \in \mathbb{F}_{13}^{\times} \smallsetminus (\mathbb{F}_{13}^{\times})^2$.
Then there does not exist any quintic form $F$ of the form \eqref{eq:quintic3q13} such that the desingularization of $V (F) \subset \mathbf{P}^2$ is a superspecial trigonal curve of genus $5$. 
\end{prop}

\begin{proof}
Put $t = 11$, $u = 3$ and
\begin{eqnarray}
\{ p_1, \ldots , p_t \} & = & \{ x^4 z, x^3 y z, x^2 y^2 z, x y^3 z, y^4 z,  x^5, x^4 y, x^3 y^2, x^2 y^3, x y^4, y^5 \}, \nonumber \\
\{ q_1, \ldots , q_u \} & = & \{ (x^2 - \epsilon y^2 ) z^3,  x (x^2 + 3 \epsilon y^2 ) z^2, y (3 x^2 + \epsilon y^2 ) z^2 \}. \nonumber 
\end{eqnarray}
Execute Enumeration Algorithm 2 given in Subsection \ref{subsec:algorithm}.
We here give an outline of our computation together with our choices of $s_1$, $s_2$, $\{ k_1, \ldots , k_{s_1} \}$, $\{ i_1, \ldots , i_{s_2} \}$, $\mathcal{A}_1$, $\mathcal{A}_2$ and a term ordering in the algorithm.

First construct the quadratic extension field $K^{\prime}:=\mathbb{F}_{13} [T] / \langle T^2 - \epsilon \rangle \cong \mathbb{F}_{169}$, where we interpret $T = \sqrt{\epsilon}$.
Compute
\[
P = \begin{pmatrix}
1 & 1 & 0 & 0 & 0\\
\sqrt{\epsilon}^{-1}& -\sqrt{\epsilon}^{-1} & 0 & 0 & 0\\
0 & 0 & 1& 0 &0\\
0 & 0 &  0& 1 & 1 \\
0 & 0 & 0 & \sqrt{\epsilon}^{-1}& -\sqrt{\epsilon}^{-1}\\
\end{pmatrix}
\]
and $P^{-1}$.
Compute $P^{(13)}$, which is the matrix obtained by taking the $13$-th power of each entry of $P$.
\begin{enumerate}
\item[{\rm (0)}]
We set $s_1:= 11$, and $( k_1, \ldots , k_{s_1} ):= ( 1, \ldots , 11)$ (we regard the $11$ coefficients $a_1$, $a_2$, $a_3$, $a_4$, $a_5$, $a_6$, $a_7$, $a_8$, $a_9$, $a_{10}$ and $a_{11}$ as indeterminates).
Let $\mathcal{A}_1:=\emptyset$.
\end{enumerate}
We proceed with the following five steps:
\begin{enumerate}
	\item[{\rm (1)}] Compute $F:= \sum_{i=1}^t a_i p_i + \sum_{j=1}^u b_j q_j$ and $h:= ( F )^{p-1}$ over $\mathbb{F}_{13} [a_1, \ldots, a_{11}] [x, y, z]$, where $a_1, \ldots, a_{11}$ are indeterminates and $(b_1,b_2,b_3)=(1,1,0)$.
	\item[{\rm (2)}] We set $s_2 := 6$, and $(i_1, \ldots , i_{s_2}) := ( 6, 7, 8, 9, 10, 11)$ (we regard the $6$ coefficients $a_6$, $a_7$, $a_8$, $a_9$, $a_{10}$ and $a_{11}$ as indeterminates).
	For the Gr\"{o}bner basis computation in the polynomial ring $\mathbb{F}_{13} [ a_1, a_2, a_3, a_4, a_5, a_6, a_7, a_8, a_9, a_{10}, a_{11}]$ in (5), we adopt the graded reverse lexicographic (grevlex) order with
\[
a_{11} \prec a_{10} \prec a_9 \prec a_8 \prec a_7 \prec a_6 \prec a_{5} \prec a_4 \prec a_3 \prec a_2 \prec a_1,
\]
	whereas for that in $\mathbb{F}_{13} [ x, y, z]$, the grevlex order with $z \prec y \prec x$ is adopted.
	Put $\mathcal{A}_2 := ( \mathbb{F}_{13} )^{\oplus 5}$.
	\item[{\rm (3)}] Substitute $\frac{X+Y}{2}$ and $\frac{X-Y}{-2\sqrt{\epsilon}}$ into $x$ and $y$ in $h$.
	Let $h^{\prime}$ denote the polynomial transformed from $h$ by the above substitution.
	\item[{\rm (4)}] Let $H^{\prime}$ be the Hasse-Witt matrix with respect to \eqref{BasisForH'}, i.e., the $(i,j)$-entry of $H^{\prime}$ ($1 \leq i,j \leq 5$) is the coefficient of the monomial $B_i B_j^{-p}$ in $h^{\prime}$.
	\item[{\rm (5)}] Compute $H:=P^{(13)} H^{\prime} P^{-1}$.
	Let $\mathcal{S}$ be the set of the $25$ entries of $H$.
	Note that $\mathcal{S} \subset \mathbb{F}_{13} [a_1, \ldots, a_{11}]$.
	For each $( c_1, c_2, c_3, c_4, c_5 ) \in \mathcal{A}_2$, conduct procedures similar to the proof of Proposition \ref{prop:Case1q7}.
	\end{enumerate}
From the outputs, we have that there does not exist any quintic form $F$ of the form \eqref{eq:quintic3q13} such that the desingularization of $V (F) \subset \mathbf{P}^2$ is a superspecial trigonal curve of genus $5$. 
\end{proof}

\subsubsection{Non-split node case (2) with $q = p=13$}

\begin{prop}\label{prop:Case4q13}
Consider the quintic form
\begin{equation}
\begin{split}
F  = &  (x^2 - \epsilon y^2 ) z^3 + (c x^4 + a_2 x^3 y + a_3 x^2 y^2 + a_4 x y^3 + a_5 y^4) z \\
& + a_6 x^5 + a_7 x^4 y + a_8 x^3 y^2 + a_9 x^2 y^3 + a_{10} x y^4 + a_{11} y^5,
\end{split}\label{eq:quintic4q13}
\end{equation}
where $a_i \in \mathbb{F}_{13}$ for $2 \leq i \leq 11$ and $c \in \{ 1, \zeta^{(13)} \}$.
Then there does not exist any quintic form $F$ of the form \eqref{eq:quintic4q13} such that the desingularization of $V (F) \subset \mathbf{P}^2$ is a superspecial trigonal curve of genus $5$. 
\end{prop}

\begin{proof}
Put $t = 11$, $u = 1$ and
\begin{eqnarray}
\{ p_1, \ldots , p_t \} & = & \{ x^4 z, x^3 y z, x^2 y^2 z, x y^3 z, y^4 z,  x^5, x^4 y, x^3 y^2, x^2 y^3, x y^4, y^5 \}, \nonumber \\
\{ q_1, \ldots , q_u \} & = & \{ (x^2 - \epsilon y^2 ) z^3 \}. \nonumber 
\end{eqnarray}
Execute Enumeration Algorithm 2 given in Subsection \ref{subsec:algorithm}.
We here give an outline of our computation together with our choices of $s_1$, $s_2$, $\{ k_1, \ldots , k_{s_1} \}$, $\{ i_1, \ldots , i_{s_2} \}$, $\mathcal{A}_1$, $\mathcal{A}_2$ and a term ordering in the algorithm.

First construct the quadratic extension field $K^{\prime}:=\mathbb{F}_{13} [T] / \langle T^2 - \epsilon \rangle \cong \mathbb{F}_{169}$, where we interpret $T = \sqrt{\epsilon}$.
Compute
\[
P = \begin{pmatrix}
1 & 1 & 0 & 0 & 0\\
\sqrt{\epsilon}^{-1}& -\sqrt{\epsilon}^{-1} & 0 & 0 & 0\\
0 & 0 & 1& 0 &0\\
0 & 0 &  0& 1 & 1 \\
0 & 0 & 0 & \sqrt{\epsilon}^{-1}& -\sqrt{\epsilon}^{-1}\\
\end{pmatrix}
\]
and $P^{-1}$.
Compute $P^{(13)}$, which is the matrix obtained by taking the $13$-th power of each entry of $P$.
\begin{enumerate}
\item[{\rm (0)}]
We set $s_1:= 11$, and $( k_1, \ldots , k_{s_1} ):= ( 1, \ldots , 11)$ (we regard the $11$ coefficients $a_1$, $a_2$, $a_3$, $a_4$, $a_5$, $a_6$, $a_7$, $a_8$, $a_9$, $a_{10}$ and $a_{11}$ as indeterminates).
Let $\mathcal{A}_1:=\emptyset$.
\end{enumerate}
We proceed with the following five steps:
\begin{enumerate}
	\item[{\rm (1)}] Compute $F:= \sum_{i=1}^t a_i p_i + \sum_{j=1}^u b_j q_j$ and $h:= ( F )^{p-1}$ over $\mathbb{F}_{13} [a_1, \ldots, a_{11}] [x, y, z]$, where $a_1, \ldots, a_{11}$ are indeterminates and $b_1=1$.
	\item[{\rm (2)}] We set $s_2 := 6$, and $(i_1, \ldots , i_{s_2}) := ( 6, 7, 8, 9, 10, 11)$ (we regard the $6$ coefficients $a_6$, $a_7$, $a_8$, $a_9$, $a_{10}$ and $a_{11}$ as indeterminates).
	For the Gr\"{o}bner basis computation in the polynomial ring $\mathbb{F}_{13} [ a_1, a_2, a_3, a_4, a_5, a_6, a_7, a_8, a_9, a_{10}, a_{11}]$ in (5), we adopt the graded reverse lexicographic (grevlex) order with
\[
a_{11} \prec a_{10} \prec a_9 \prec a_8 \prec a_7 \prec a_6 \prec a_{5} \prec a_4 \prec a_3 \prec a_2 \prec a_1,
\]
	whereas for that in $\mathbb{F}_{13} [ x, y, z]$, the grevlex order with $z \prec y \prec x$ is adopted.
	Put $\mathcal{A}_2 := \{ (0,0), (1,0), (0,1), (1,1), (1, \zeta^{(13)}) \} \oplus (\mathbb{F}_{13})^{\oplus 3}$.
	\item[{\rm (3)}] Substitute $\frac{X+Y}{2}$ and $\frac{X-Y}{-2\sqrt{\epsilon}}$ into $x$ and $y$ in $h$.
	Let $h^{\prime}$ denote the polynomial transformed from $h$ by the above substitution.
	\item[{\rm (4)}] Let $H^{\prime}$ be the Hasse-Witt matrix with respect to \eqref{BasisForH'}, i.e., the $(i,j)$-entry of $H^{\prime}$ ($1 \leq i,j \leq 5$) is the coefficient of the monomial $B_i B_j^{-p}$ in $h^{\prime}$.
	\item[{\rm (5)}] Compute $H:=P^{(13)} H^{\prime} P^{-1}$.
	Let $\mathcal{S}$ be the set of the $25$ entries of $H$.
	Note that $\mathcal{S} \subset \mathbb{F}_{13} [a_1, \ldots, a_{11}]$.
	For each $( c_1, c_2, c_3, c_4, c_5 ) \in \mathcal{A}_2$, conduct procedures similar to the proof of Proposition \ref{prop:Case1q7}.
	\end{enumerate}
From the outputs, we have that there does not exist any quintic form $F$ of the form \eqref{eq:quintic4q13} such that the desingularization of $V (F) \subset \mathbf{P}^2$ is a superspecial trigonal curve of genus $5$. 
\end{proof}

\subsubsection{Non-split node case (3) with $q = p=13$}

\begin{prop}\label{prop:Case4-2q13}
Consider the quintic form
\begin{equation}
\begin{split}
F  = &  (x^2 - \epsilon y^2 ) z^3 + a_6 x^5 + a_7 x^4 y + a_8 x^3 y^2 + a_9 x^2 y^3 + a_{10} x y^4 + a_{11} y^5,
\end{split}\label{eq:quintic4-2q13}
\end{equation}
where $a_i \in \mathbb{F}_{13}$ for $6 \leq i \leq 11$.
Then there does not exist any quintic form $F$ of the form \eqref{eq:quintic4-2q13} such that the desingularization of $V (F) \subset \mathbf{P}^2$ is a superspecial trigonal curve of genus $5$. 
\end{prop}

\begin{proof}
Put $t = 6$, $u = 1$ and
\begin{eqnarray}
\{ p_1, \ldots , p_t \} & = & \{ x^5, x^4 y, x^3 y^2, x^2 y^3, x y^4, y^5 \}, \nonumber \\
\{ q_1, \ldots , q_u \} & = & \{ (x^2 - \epsilon y^2 ) z^3 \}. \nonumber 
\end{eqnarray}
Execute Enumeration Algorithm 2 given in Subsection \ref{subsec:algorithm}.
We here give an outline of our computation together with our choices of $s_1$, $s_2$, $\{ k_1, \ldots , k_{s_1} \}$, $\{ i_1, \ldots , i_{s_2} \}$, $\mathcal{A}_1$, $\mathcal{A}_2$ and a term ordering in the algorithm.

First construct the quadratic extension field $K^{\prime}:=\mathbb{F}_{13} [T] / \langle T^2 - \epsilon \rangle \cong \mathbb{F}_{169}$, where we interpret $T = \sqrt{\epsilon}$.
Compute
\[
P = \begin{pmatrix}
1 & 1 & 0 & 0 & 0\\
\sqrt{\epsilon}^{-1}& -\sqrt{\epsilon}^{-1} & 0 & 0 & 0\\
0 & 0 & 1& 0 &0\\
0 & 0 &  0& 1 & 1 \\
0 & 0 & 0 & \sqrt{\epsilon}^{-1}& -\sqrt{\epsilon}^{-1}\\
\end{pmatrix}
\]
and $P^{-1}$.
Compute $P^{(13)}$, which is the matrix obtained by taking the $13$-th power of each entry of $P$.
\begin{enumerate}
\item[{\rm (0)}]
We set $s_1:= 6$, and $( k_1, \ldots , k_{s_1} ):= ( 6, \ldots , 11)$ (we regard the $6$ coefficients $a_6$, $a_7$, $a_8$, $a_9$, $a_{10}$ and $a_{11}$ as indeterminates).
Let $\mathcal{A}_1:=\emptyset$.
\end{enumerate}
We proceed with the following five steps:
\begin{enumerate}
	\item[{\rm (1)}] Compute $F:= \sum_{i=1}^t a_i p_i + \sum_{j=1}^u b_j q_j$ and $h:= ( F )^{p-1}$ over $\mathbb{F}_{13} [a_6, \ldots, a_{11}] [x, y, z]$, where $a_6, \ldots, a_{11}$ are indeterminates and $b_1=1$.
	\item[{\rm (2)}] We set $s_2 := 5$, and $(i_1, \ldots , i_{s_2}) := ( 7, 8, 9, 10, 11)$ (we regard the $5$ coefficients $a_7$, $a_8$, $a_9$, $a_{10}$ and $a_{11}$ as indeterminates).
	For the Gr\"{o}bner basis computation in the polynomial ring $\mathbb{F}_{13} [ a_6, a_7, a_8, a_9, a_{10}, a_{11}]$ in (5), we adopt the graded reverse lexicographic (grevlex) order with
\[
a_{11} \prec a_{10} \prec a_9 \prec a_8 \prec a_7 \prec a_6 ,
\]
	whereas for that in $\mathbb{F}_{13} [ x, y, z]$, the grevlex order with $z \prec y \prec x$ is adopted.
	Put $\mathcal{A}_2 := \mathbb{F}_{13}$.
	\item[{\rm (3)}] Substitute $\frac{X+Y}{2}$ and $\frac{X-Y}{-2\sqrt{\epsilon}}$ into $x$ and $y$ in $h$.
	Let $h^{\prime}$ denote the polynomial transformed from $h$ by the above substitution.
	\item[{\rm (4)}] Let $H^{\prime}$ be the Hasse-Witt matrix with respect to \eqref{BasisForH'}, i.e., the $(i,j)$-entry of $H^{\prime}$ ($1 \leq i,j \leq 5$) is the coefficient of the monomial $B_i B_j^{-p}$ in $h^{\prime}$.
	\item[{\rm (5)}] Compute $H:=P^{(13)} H^{\prime} P^{-1}$.
	Let $\mathcal{S}$ be the set of the $25$ entries of $H$.
	Note that $\mathcal{S} \subset \mathbb{F}_{13} [a_6, \ldots, a_{11}]$.
	For each $c_6  \in \mathcal{A}_2$, conduct procedures similar to the proof of Proposition \ref{prop:Case1q7}.
	\end{enumerate}
From the outputs, we have that there does not exist any quintic form $F$ of the form \eqref{eq:quintic4-2q13} such that the desingularization of $V (F) \subset \mathbf{P}^2$ is a superspecial trigonal curve of genus $5$. 
\end{proof}

\subsubsection{Cusp case with $q = p=13$}

\begin{prop}\label{prop:Case5q13}
Consider the quintic form
\begin{equation}
\begin{split}
F  = &  x^2 z^3 + a_1 y^3 z^2 + (a_2 x^4 + a_3 x^3 y + a_4 x^2 y^2 + b_1 x y^3 + a_5 y^4) z \\
& + a_6 x^5 + a_7 x^4 y + a_8 x^3 y^2 + a_9 x^2 y^3 + b_2 x y^4 + a_{10} y^5,
\end{split}\label{eq:quintic5q13}
\end{equation}
where $a_i \in \mathbb{F}_{13}$ for $1 \leq i \leq 10$ with $a_1 \neq 0$ and $(b_1, b_2) \in \{ (0,0), (1,0), (0,1), (1,1) \}$.
Then there does not exist any quintic form $F$ of the form \eqref{eq:quintic5q13} such that the desingularization of $V (F) \subset \mathbf{P}^2$ is a superspecial trigonal curve of genus $5$. 
\end{prop}

\begin{proof}
Put $t = 10$, $u = 3$ and
\begin{eqnarray}
\{ p_1, \ldots , p_t \} & = & \{ y^3 z^2, x^4 z, x^3 y z, x^2 y^2 z, y^4 z,  x^5, x^4 y, x^3 y^2, x^2 y^3, y^5 \}, \nonumber \\
\{ q_1, \ldots , q_u \} & = & \{ x y^3 z, x y^4,  x^2 z^3 \}. \nonumber 
\end{eqnarray}
For each $(b_1, b_2) \in \{ (0,0), (1,0), (0,1), (1,1) \}$, execute Enumeration Algorithm 1 given in Subsection \ref{subsec:algorithm}.
We here give an outline of our computation together with our choices of $s_1$, $s_2$, $\{ k_1, \ldots , k_{s_1} \}$, $\{ i_1, \ldots , i_{s_2} \}$, $\mathcal{A}_1$, $\mathcal{A}_2$ and a term ordering in the algorithm.
\begin{enumerate}
\item[{\rm (0)}]
We set $s_1:= 9$, and $( k_1, \ldots , k_{s_1} ):= ( 2, \ldots , 10)$ (we regard the $9$ coefficients $a_2$, $a_3$, $a_4$, $a_5$, $a_6$, $a_7$, $a_8$, $a_9$ and $a_{10}$ as indeterminates).
Let $\mathcal{A}_1:=\mathbb{F}_{13}^{\times}$.
\end{enumerate}
For each $c_1 \in \mathcal{A}_1$, we proceed with the following three steps:
\begin{enumerate}
	\item[{\rm (1)}] Substitute $c_1$ into $a_1$ in $F$, and compute $F:= \sum_{i=1}^t a_i p_i + \sum_{j=1}^u b_j q_j$ and $h:= ( F )^{p-1}$ over $\mathbb{F}_{13} [a_2, \ldots, a_{10}] [x, y, z]$, where $a_2, \ldots, a_{10}$ are indeterminates and $b_3=1$.
	\item[{\rm (2)}] We set $s_2 := 7$, and $(i_1, \ldots , i_{s_2}) := ( 2, 3, 4, 6, 7, 8, 9)$ (we regard the $7$ coefficients $a_2$, $a_3$, $a_4$, $a_6$, $a_7$, $a_8$ and $a_9$ as indeterminates).
	For the Gr\"{o}bner basis computation in the polynomial ring $\mathbb{F}_{13} [a_2, a_3, a_4, a_5, a_6, a_7, a_8, a_9, a_{10}]$ in (3), we adopt the graded reverse lexicographic (grevlex) order with
\[
a_{10} \prec a_9 \prec a_8 \prec a_7 \prec a_6 \prec a_5 \prec a_4 \prec a_3 \prec a_2,
\]
	whereas for that in $\mathbb{F}_{13} [ x, y, z]$, the grevlex order with $z \prec y \prec x$ is adopted.
	Put $\mathcal{A}_2 := (\mathbb{F}_{13})^{\oplus 2}$.
	\item[{\rm (3)}] Let $\mathcal{S} \subset \mathbb{F}_{13} [a_2, \ldots, a_{10}]$ be the set of the coefficients of the $25$ monomials in $h$, given in Corollary \ref{Criterion_ssp_non-split}.
	For each $( c_5, c_{10} ) \in \mathcal{A}_2$, conduct procedures similar to the proof of Proposition \ref{prop:Case1q7}.
\end{enumerate}
From the outputs, we have that there does not exist any quintic form $F$ of the form \eqref{eq:quintic5q13} such that the desingularization of $V (F) \subset \mathbf{P}^2$ is a superspecial trigonal curve of genus $5$.
\end{proof}

\subsection{Remarks on our implementation to prove main theorems}\label{subsec:imple}
We implemented and executed computations in the proofs of Propositions \ref{prop:Case1q7} -- \ref{prop:Case5q13}, including Enumeration Algorithms 1 and 2 given in Subsection \ref{subsec:algorithm}, over Magma V2.22-7 \cite{Magma}, \cite{MagmaHP}.
All our computations were conducted on a computer with ubuntu 16.04 LTS OS at 3.40 GHz CPU (Intel Core i7-6700) and 15.6 GB memory.
All the source codes and the log files are available at the web page of the first author \cite{HPkudo}.
All the computations for our enumeration were terminated in about 4.4 days in total.
In our implementations, we computed Gr\"{o}bner bases with the Magma function \texttt{GroebnerBasis}, which adopts the $F_4$ algorithm \cite{F4}.

%% file: section5.tex
\section{Automorphism groups}\label{sec:aut}

This section studies the isomorphisms and the automorphism groups of superspecial trigonal curves of genus $5$, and enumerate superspecial trigonal curves over $\mathbb{F}_{11}$ by Galois cohomology theory.

\subsection{Isomorphisms}

\begin{prop}\label{prop:isom}
\begin{enumerate}
\item[${\rm (I)}$] There are precisely $4$ $\mathbb{F}_{11}$-isomorphism classes of superspecial trigonal curves of genus $5$ over $\mathbb{F}_{11}$.
The four isomorphism classes are the desingularizations of the curves $C_i=V (F_i) \subset \mathbf{P}^2$ given by
\begin{eqnarray}
F_1 & = & x y z^3 + x^5 + y^5 , \nonumber \\
F_2 & = & x y z^3 + 2 x^5 + y^5 , \nonumber \\
F_3 & = & x y z^3 + 3 x^5 + y^5 , \nonumber \\
F_4 & = & (x^2 - 2 y^2) z^3  + x^5 + 9 x^3 y^2 + 9 x y^4. \nonumber
\end{eqnarray}
\item[${\rm (I\hspace{-.1em}I)}$] There is a unique $\overline{\mathbb{F}_{11}}$-isomorphism class of superspecial trigonal curves of genus $5$ over $\mathbb{F}_{11}$.
The unique isomorphism class is the desingularization of the curve $C^{\rm (alc)} = V (F) \subset \mathbf{P}^2$ given by $F =  x y z^3 + x^5 + y^5$.
\end{enumerate}
\end{prop}

\begin{proof}
\begin{enumerate}
\item[(I)]
By Theorem \ref{MainTheorem}, any superspecial trigonal curve of genus $5$ over $\F_{11}$ is $\F_{11}$-isomorphic to the desingularization of the quintic in $\mathbf{P}^2$ defined by
\begin{equation}
x y z^3 + a_1 x^5 + a_2 y^5  = 0 \label{sscurve_F11_g5}
\end{equation}
for some $a_1, a_2 \in \mathbb{F}_{11}^\times$, or
\begin{equation}
 (x^2 - \epsilon y^2) z^3 + a x^5 + b x^4 y + (9 a) x^3 y^2 + (4 b) x^2 y^3 + ( 9 a ) x y^4 + (3 b) y^5 =0 \label{sscurve_F11_2_g5}
\end{equation}
for some $(a, b) \in (\mathbb{F}_{11})^{\oplus 2} \smallsetminus \{ (0,0 )\}$.
Let $\mathcal{F}_0$ be the set of quintic forms of the forms \eqref{sscurve_F11_g5} and \eqref{sscurve_F11_2_g5}.
Here we set $K:=\mathbb{F}_{11}$.
Recall from Lemma \ref{CharacterizationTrigonal} (3) that $V(F) \cong V (F^{\prime})$ over $K$ for $F, F^{\prime} \in \mathcal{F}_0$ if and only if their desingularizations are isomorphic to each other over $K$.
Thus, it suffices to compute a set $\mathcal{F} \subset \mathcal{F}_0$ such that for each $F, F^{\prime} \in \mathcal{F}$ with $F \neq F^{\prime}$, the two curves $V (F)$ and $V(F^{\prime})$ are not isomorphic to each other over $K$.
More specifically, $V ( F )$ and $V(F^{\prime})$ are not isomorphic to each other over $K$ if and only if there dose not exist any element $\mu \in \mathrm{Aut}_K ( \mathbf{P}^2 )$ such that $\mu ( V ( F ) ) = V ( F^{\prime} )$.
Here, an element $\mu \in \mathrm{Aut}_K ( \mathbf{P}^2 )$ with $\mu ( V ( F ) ) = V ( F^{\prime} )$ is given as a matrix $M \in \mathrm{GL}_3 ( K )$ such that $M \cdot F = \lambda F^{\prime}$ for some $\lambda \in K^{\times}$.
Since the action of such an $M$ stabilizes the singular point $(0:0:1)$ on $V (F)$ and $V(F^{\prime})$, its $(1,3)$, $(2,3)$ and $(3,3)$ entries are $0$, $0$ and $1$ respectively.
Thus, executing the following procedures (a) -- (c) for each pair $( F, F^{\prime} )$ of elements in $\mathcal{F}_0$ with $F \neq F^{\prime}$, we obtain such a set $\mathcal{F}$:
\begin{enumerate}
\item We set
\[
M:=
\begin{pmatrix}
a&b&0\\
c&d&0\\
e&f&1
\end{pmatrix},
\]
where $a$, $b$, $c$, $d$, $e$ and $f$ are indeterminates.
\item Computing $F^{\prime \prime} := M \cdot F - \lambda F^{\prime}$, construct a multivariate system over $K$ derived from the equation $F^{\prime \prime}=0$ together with $\lambda \cdot g \cdot \mathrm{det}(M) = \lambda g ( a d - b c ) = 1$.
Here $\lambda$ and $g$ are extra indeterminates.
\item Decide whether the system constructed in (b) has a root over $K$ or not.
If the system has a root over $K$, the two curves $V (F)$ and $V (F^{\prime})$ are isomorphic to each other over $K$.
Otherwise $V (F)$ and $V (F^{\prime})$ are not isomorphic to each other over $K$.
\end{enumerate}
To proceed with the above procedures (a) -- (c), for example use a Gr\"{o}bner basis (see the web page of the first author \cite{HPkudo}, for a code by Magma).
The set $\mathcal{F}$ computed by our implementation over Magma consists of $F_i$ for $1 \leq i \leq 4$ in the statement.
\item[(I\hspace{-.1em}I)] This is proved in a way similar to (I).
Specifically, replacing $K$ by $\overline{\mathbb{F}_{11}}$, execute the same procedures as in (I) for $\mathcal{F}_0 := \{ F_i : 1 \leq i \leq 4 \}$.
\end{enumerate}
\end{proof}

\begin{rem}
The number of $\mathbb{F}_{121}$-rational points on each curve given in Proposition \ref{prop:isom} (1) is $\# C_1= 232$, $\# C_2 = 122$, $\# C_3 = 122$, $\# C_4 = 232$.
Hence $C_1$ and $C_4$ are maximal curves over $\mathbb{F}_{121}$.
\end{rem}

\subsection{Automorphism}

\begin{prop}\label{prop:Aut}
\begin{enumerate}
\item[${\rm (I)}$] Let $C_i = V ( F_i )$ denote the superspecial trigonal curve of genus $5$ over $\mathbb{F}_{11}$ defined by $F_i$ for each $1 \leq i \leq 4$.
Here each $F_i$ is given in {\rm Proposition} $\ref{prop:isom}$ {\rm (I)}.
Then we have the following isomorphisms:

\begin{tabular}{clcl}
$(1)$ & $\Aut_{\mathbb{F}_{11}} ( C_1 ) \cong \mathrm{D}_5$,  & 
$(2)$ & $\Aut_{\mathbb{F}_{11}} ( C_2 ) \cong \mathrm{C}_5$, \\[2mm]
$(3)$ & $\Aut_{\mathbb{F}_{11}} ( C_3 ) \cong \mathrm{C}_5$, &
$(4)$ & $\Aut_{\mathbb{F}_{11}} ( C_4 ) \cong \mathrm{C}_2$,
\end{tabular}

\noindent where ${\rm D}_t$ and $\mathrm{C}_t$ denote the dihedral group and the cyclic group of degree $t$ for each $t$.

\item[${\rm (I\hspace{-.1em}I)}$] Let $C^{\rm (alc)} = V ( F )$ denote the $\overline{\mathbb{F}_{11}}$-isomorphism class of superspecial trigonal curves of genus $5$ over $\mathbb{F}_{11}$ defined by $F$.
Here $F$ is given in {\rm Proposition} $\ref{prop:isom}$ {\rm (I\hspace{-.1em}I)}.
Then we have an isomorphism $\Aut ( C^{\rm (alc)} ) \cong \mathrm{C}_3 \times \mathrm{D}_5$.
\end{enumerate}
\end{prop}

\begin{proof}
\begin{enumerate}
\item[(I)]
We prove only the statement (1) since the other cases (2) -- (4) are proved in ways similar to this.
To simplify the notation, we set $F:=F_1$ and $C:=C_1=V(F)$.
Putting
\[
G_K := \{ M \in \mathrm{GL}_3 (K) : M \cdot F = \lambda F \mbox{ for some } \lambda \in K^{\times} \} 
\]
with $K := \mathbb{F}_{11}$, we have $\mathrm{Aut}_K(C) \cong  G_K / \sim $, where $M \sim c M$ for some $c \in K^{\times}$.
Note that the $(1,3)$, $(2,3)$ and $(3,3)$ entries of each element in $G_K$ are $0$, $0$, and $1$ respectively.
Determining the set $G_K$ is reduced into solving a multivariate system over $K$ as follows:
Here we set
\[
M:=
\begin{pmatrix}
a&b&0\\
c&d&0\\
e&f&1
\end{pmatrix},
\]
where $a$, $b$, $c$, $d$, $e$ and $f$ are indeterminates.
Now we have a multivariate system over $\mathbb{F}_{11}$ derived from the equation $M \cdot F = \lambda F$ together with $\lambda \cdot g \cdot \mathrm{det}(M) = \lambda g ( a d - b c ) = 1$, where $\lambda$ and $g$ are extra indeterminates.
To solve this, for example use a Gr\"{o}bner basis (see the web page of the first author \cite{HPkudo}, for a code by Magma).
As a result, we have that $G_K / \sim$ is generated by
\[
a:=
\begin{pmatrix}
0&1&0\\
1&0&0\\
0&0&1
\end{pmatrix}
\mbox{ and }
b:=
\begin{pmatrix}
3&0&0\\
0&4&0\\
0&0&1
\end{pmatrix},
\]
whose orders are $2$ and $5$, respectively.
Hence, the homomorphism defined by $a \mapsto (1, 5) (2, 4) $ and $b \mapsto (1, 2, 3, 4, 5)$ gives an isomorphism between $G_K / \sim$ and $\mathrm{D}_5$, where we identify $\mathrm{D}_5$ with the subgroup of ${\rm S}_5$ generated by the permutations $(1, 5) (2, 4) $ and $(1, 2, 3, 4, 5)$.
\item[(I\hspace{-.1em}I)] This is proved in a way similar to (I).
Specifically, we have that $G_K / \sim$ with $K = \overline{\mathbb{F}_{11}}$ is generated by
\[
a:=
\begin{pmatrix}
4&0&0\\
0&3&0\\
0&0&1
\end{pmatrix}, \quad
b:=
\begin{pmatrix}
0&4&0\\
3&0&0\\
0&0&1
\end{pmatrix}, \quad
\mbox{and }
c:=
\begin{pmatrix}
\zeta^{80}&0&0\\
0&\zeta^{80}&0\\
0&0&1
\end{pmatrix},
\]
whose orders are $5$, $2$ and $3$, respectively.
Here $\zeta$ is a root of $t^2 + 7 t + 2$, which is a primitive element of $\mathbb{F}_{121}$.
Hence, the homomorphism defined by $a \mapsto (4,5,6,7,8)$, $b \mapsto (4,8) (5,7)$ and $c \mapsto (1, 2, 3)$ gives an isomorphism between $G_K / \sim$ and $ \mathrm{C}_3 \times \mathrm{D}_5$, where we identify $ \mathrm{C}_3 \times \mathrm{D}_5$ with the subgroup of ${\rm S}_8$ generated by the permutations $(4,5,6,7,8)$, $(4,8) (5,7)$ and $(1, 2, 3)$.
\end{enumerate}
\end{proof}

In Table \ref{tab:ACAut}, we summarize the results in Proposition \ref{prop:Aut}.
Here an $\mathbb{F}_{11}$-form of $C$ is a (nonhyperelliptic superspecial trigonal) curve $C^{\prime}$ over $\mathbb{F}_{11}$ such that $C \cong C^{\prime}$ over $\overline{\mathbb{F}_{11}}$, where $\mathrm{Aut}(C)$ denotes the automorphism group over $\overline{\mathbb{F}_{11}}$.

\renewcommand{\arraystretch}{1.3}
\begin{table}[H]
\begin{center}
\begin{tabular}{ccc|ccc} \hline
Superspecial trigonal & \multirow{2}{*}{$\Aut(C)$} & \multirow{2}{*}{$\# \Aut(C)$} & $\mathbb{F}_{11}$-forms $C^{\prime}$ & \multirow{2}{*}{$\Aut_{\mathbb{F}_{11}}(C^{\prime})$} & \multirow{2}{*}{$\# \Aut_{\mathbb{F}_{11}}(C^{\prime})$} \\ 
curves $C$ over $\overline{\mathbb{F}_{11}}$ & & & of $C$ & & \\ \hline
\multirow{4}{*}{$C^{\rm (alc)}$} & \multirow{4}{*}{${\rm C}_3 \times {\rm D}_5$} & \multirow{4}{*}{$30$} & $C_1$ & ${\rm D}_5$ & $10$ \\
 & & & $C_2$ & ${\rm C}_5$ & $5$ \\
 & & & $C_3$ & ${\rm C}_5$ & $5$ \\
 & & & $C_4$ & ${\rm C}_2$ & $2$ \\ \hline
\end{tabular}
\caption{The automorphism group $\mathrm{Aut} (C^{\rm (alc)}):=\mathrm{Aut}_{\overline{\mathbb{F}_{11}}} (C^{\rm (alc)})$ of the superspecial trigonal curve $C^{\rm (alc)}:=V(F) = V (F_1)$ over $\overline{\mathbb{F}_{11}}$ and the automorphism groups $\mathrm{Aut}_{\mathbb{F}_{11}} (C_i)$ of the superspecial trigonal curves 
$C_i:=V(F_i)$ over $\mathbb{F}_{11}$ for $1 \leq i \leq 4$.
Here $F$ and $F_i$ for $1 \leq i \leq 4$ are defined in {\rm Proposition} $\ref{prop:isom}$.
}\label{tab:ACAut}
\end{center}
\end{table}

\subsection{Compatibility with Galois cohomology}
In this subsection, we show that our enumeration of superspecial trigonal curves over the prime field $\mathbb{F}_{11}$ is compatible with that by Galois cohomology together with our result over the algebraic closure.

We denote by $\Gamma$ the absolute Galois group $\mathrm{Gal}(\overline{\mathbb{F}_{11}} / \mathbb{F}_{11})$.
Let $C$ be a nonhyperelliptic superspecial trigonal curve over $\mathbb{F}_{11}$, and $\mathrm{Aut}(C)$ its automorphism group over the algebraic closure $\overline{\mathbb{F}_{11}}$.
Let $\sigma$ be the Frobenius on $\mathrm{Aut}(C)$.
For two elements $a$ and $b$ of $\mathrm{Aut}(C)$, they are said to be $\sigma$-{\it conjugate} if $a = g^{-1}bg^\sigma$ for some $g \in \mathrm{Aut}(C )$.
Then one has the group isomorphism
\begin{equation}\label{GaloisCoh-Aut}
H^1 ( \Gamma, \mathrm{Aut} ( C ) ) \cong \mathrm{Aut}(C ) / \sigma\text{-conjugacy} .
\end{equation}
Here, it is known that $H^1 ( \Gamma, \mathrm{Aut} (C ) )$ parametrizes $\mathbb{F}_{11}$-forms of $C$, see \cite[Chap.III, \S 1.1, Prop. 1]{S}.
Let $a \in \mathrm{Aut}( C )$, and let $C^{(a)}$ denote the $\F_{11}$-form associated to $a$ via the isomorphism \eqref{GaloisCoh-Aut}.
If $Frob$ is the Frobenius map on $C$, then the Frobenius on $C^{(a)}$ is given by $a (Frob)$ via an isomorphism from $C \otimes \overline{\mathbb{F}_{11}}$ to $C^{(a)}\otimes \overline{\mathbb{F}_{11}}$.
Then we have
\begin{eqnarray}
\mathrm{Aut}_{\mathbb{F}_{11}} (C^{(a)}) & \simeq & \{ g \in \mathrm{Aut}( C ) \mid g (a (Frob)) = (a (Frob)) g \} \nonumber \\
& = & \{ g \in \mathrm{Aut}( C ) \mid a = g^{-1}ag^\sigma\} \nonumber \\
& = & \sigma\text{-}{\rm Stab}_{\Aut (C )}(a), \nonumber
\end{eqnarray}
where $\simeq$ is a bijection and $\sigma\text{-}{\rm Stab}_{\Aut (C )}(a)$ denotes the $\sigma$-stabilizer group of $a$
in $\mathrm{Aut}(C)$.
Then, we have the following:
\begin{prop}
With notation as above, we have $|\mathrm{Aut}(C ) / \sigma\mbox{\rm -conjugacy}| = 4$.
Moreover, the orders of the $\sigma$-stabilizer groups are $10$, $5$, $5$, and $2$.
\end{prop}
\begin{proof}
Let $F = x y z^3 + x^5 + y^5$.
Recall from Proposition \ref{prop:isom} (II) that $F$ gives the singular model of a representative of the unique $\overline{\mathbb{F}_{11}}$-isomorphism class of superspecial trigonal curves over $\mathbb{F}_{11}$.
As in the proof of Proposition \ref{prop:Aut} (II), we put
\[
G_K := \{ M \in \mathrm{GL}_3 (K) : M \cdot F = \lambda F \mbox{ for some } \lambda \in K^{\times} \}
\]
with $K= \overline{\mathbb{F}_{11}}$, and we write $M \sim c M$ for $c \in K^{\times}$.
Recall from the proof of Proposition \ref{prop:Aut} (II) that the group $G_K / \sim$, which is isomorphic to $\mathrm{Aut}(C )$, is generated by
\[
a:=
\begin{pmatrix}
4&0&0\\
0&3&0\\
0&0&1
\end{pmatrix}, \quad
b:=
\begin{pmatrix}
0&4&0\\
3&0&0\\
0&0&1
\end{pmatrix}, \quad
\mbox{and }
c:=
\begin{pmatrix}
\zeta^{80}&0&0\\
0&\zeta^{80}&0\\
0&0&1
\end{pmatrix}
\]
whose orders are $5$, $2$ and $3$, respectively.
Here $\zeta$ is a root of $t^2 + 7 t + 2$, which is a primitive element of $\mathbb{F}_{121}$.
To compute $\mathrm{Aut}(C ) / \sigma\mbox{\rm -conjugacy}$, we conduct the following:
For each pair $(M, M^{\prime})$ of two distinct elements $M$ and $M^{\prime}$ in the group $\langle a, b, c \rangle = G_K / \sim$, decide whether $M$ and $M^{\prime}$ are $\sigma$-conjugate, i.e., $M = g^{-1}M^{\prime}g^\sigma$ for some $g \in \langle a, b, c \rangle$, or not.
This can be easily decided by a straightforward computation, e.g., brute force on all elements in $g \in \langle a, b, c \rangle$.
As a result, we can take the following four matrices as representatives of $\sigma$-conjugacy classes:
\[
g_1:=
\begin{pmatrix}
1&0&0\\
0&1&0\\
0&0&1
\end{pmatrix}, \quad
g_2:=
\begin{pmatrix}
3&0&0\\
0&4&0\\
0&0&1
\end{pmatrix}, \quad
g_3:=
\begin{pmatrix}
9&0&0\\
0&5&0\\
0&0&1
\end{pmatrix}, \quad
\mbox{and }
g_4:=
\begin{pmatrix}
0&\zeta^{64}&0\\
\zeta^{16}&0&0\\
0&0&1
\end{pmatrix}.
\]
By a computation similar to the computation of $\mathrm{Aut}(C ) / \sigma\mbox{\rm -conjugacy}$, the $\sigma$-stabilizer group of each $g_i$ is determined as follows:
\begin{enumerate}
\item The $\sigma$-stabilizer group of $g_1$ has order $10$, and is generated by
\[
\begin{pmatrix}
3&0&0\\
0&4&0\\
0&0&1
\end{pmatrix},\\
\begin{pmatrix}
0&3&0\\
4&0&0\\
0&0&1
\end{pmatrix},\\
\begin{pmatrix}
0&4&0\\
3&0&0\\
0&0&1
\end{pmatrix},\\
\begin{pmatrix}
0&5&0\\
9&0&0\\
0&0&1
\end{pmatrix},\\ \mbox{ and }
\begin{pmatrix}
0&9&0\\
5&0&0\\
0&0&1
\end{pmatrix},
\]
whose orders are $5$, $2$, $2$, $2$ and $2$, respectively.
\item The $\sigma$-stabilizer group of $g_2$ has order $5$, and is generated by $g_2$.
\item The $\sigma$-stabilizer group of $g_3$ has order $5$, and is generated by $g_3$.
Note that this group is the same as the $\sigma$-stabilizer group of $g_2$.
\item The $\sigma$-stabilizer group of $g_4$ has order $5$, and is generated by $b$, which is a generator of $G_K / \sim$.
\end{enumerate}
\end{proof}
We see that $|\mathrm{Aut}(C ) / \sigma\mbox{\rm -conjugacy}|$ coincides with the number of $\mathbb{F}_{11}$-isomorphism classes of superspecial trigonal curves of genus $5$ over $\mathbb{F}_{11}$, see Proposition \ref{prop:isom} (I).
Moreover, the orders of the $\sigma$-stabilizer groups also coincide with those of automorphism groups over $\mathbb{F}_{11}$ in Proposition \ref{prop:Aut} (I).
These support the correctness of our computational enumeration over $\mathbb{F}_{11}$ in Theorem \ref{MainTheorem} and Propositions \ref{prop:isom} (I) and \ref{prop:Aut} (I).

%% file: section6.tex
\appendix

\section{Testing the irreducibility for a given quintic form}\label{sec:irr}

In this appendix, we describe a simple method to test the irreducibility for a given quintic form in $\mathbb{F}_q[x, y, z]$.
Specifically, testing the irreducibility is reduced into testing the (non-)existence of roots of multivariate systems.
This method is used in the proofs of Propositions \ref{prop:Case2q11} and \ref{prop:Case4q11}, where we show the irreducibility for some quintic forms in $\mathbb{F}_{11} [x,y,z]$.
Note that we do not write down this simple method in algorithmic format since the irreducibility test for polynomials is not the main subject of this paper.
By the same reason, let us not refer to any other method, while various algorithms for the irreducibility test have been already proposed.

Let $f$ be a quintic form in $\mathbb{F}_{q} [x,y,z]$ such that the coefficient of $x^5$ in $f$ is $1$.
In the following, we consider the irreducibility over $K = \overline{\mathbb{F}_q}$ (resp.\ $K =\mathbb{F}_{q^s}$ for some $s \geq 0$).
It is clear that the quintic form $f$ is reducible over $K$ if and only if $f$ is factored into a product of a quadratic form and a cubic form over $K$, or a linear form and a quartic form.
If $f$ is factored into a product of a quadratic form $g_2$ and a cubic form $g_3$, there exists $(b_1, \ldots , b_{14}) \in K^{\oplus 14}$ with $f = g_2 g_3$ such that
\begin{eqnarray}
g_2 & = & x^2 + b_1 x y + b_2 y^2 + b_3 x z + b_4 y z + b_5 z^2, \nonumber \\
g_3 & = & x^3 + b_6 x^2 y + b_7 x y^2 + b_8 y^3 + b_9 x^2 z + b_{10} x y z + b_{11} y^2 z + b_{12} x z^2 + b_{13} y z^2 + b_{14} z^3. \nonumber
\end{eqnarray}
In other word, the multivariate system derived from the coefficients of monomials in $f - g_2 g_3$ with respect to $x$, $y$ and $z$ has a root over $K$, where we regard $b_i$'s as indeterminates.
One can check this, for example, by computing a reduced Gr\"{o}bner basis of the ideal generated by the coefficients (resp.\ the coefficients and $\{ b_i^{q^s} - b_i : 1 \leq i \leq 14 \}$) in $\mathbb{F}_q [b_1, \ldots , b_{14}]$.
If the basis is $\{ 1 \}$, then the system has no root over $K$, otherwise it has a root over $K$.
Thus, computing the reduced Gr\"{o}bner basis with respect to some term ordering, we can decide whether $f$ is factored into a product of a quadratic form and a cubic form over $K$.
Similarly, one can test whether $f$ is factored into a product of a linear form and a cubic form over $K$.
In this way, we can test the irreducibility for $f$.

Note that this method can be applied to more general cases: polynomials of arbitrary degree in the polynomial ring of $n$ variables over $\mathbb{F}_q$ (while the computational cost will become much higher).